\documentclass[11pt]{article}
\usepackage[english]{babel}
\usepackage[utf8]{inputenc}
\usepackage{amsmath, amsfonts, amsthm, amssymb, amscd,enumerate, mathrsfs}
\usepackage{graphicx, color}
\usepackage{multirow}
\usepackage[mathcal]{euscript}
\newtheorem{theorem}{Theorem}[section]
\newtheorem{prop}{Proposition}[section]
\newtheorem{lemma}{Lemma}[section]

\newtheorem{nb}{Remark}[section]
\newtheorem{example}{Example}[section]
\newtheorem{Def}{Definition}[section]

\numberwithin{equation}{section}
\usepackage{tikz-cd}
\usepackage{tikz}
\usepackage{float}
\usepackage[labelformat=empty]{caption}
\usepackage{booktabs}
\begin{document}
\title{On the quadratic cone of $ \mathbb{R}_3$}
\date{}
\author{Cinzia Bisi, \, Antonino De Martino }
\maketitle
\begin{abstract}
In this paper we study the following type of functions $f: \mathcal{Q}_{\mathbb{R}_{3}} \to \mathbb{R}_{3}$, where $ \mathcal{Q}_{\mathbb{R}_3}$ is the quadratic cone of the algebra $\mathbb{R}_{3}$. From the fact that it is possible to write the algebra $ \mathbb{R}_{3}$ as a direct sum of quaternions, we get the observation that it is possible to find a clever representation for $ \mathcal{Q}_{\mathbb{R}_3}$. By using this result a slice regular theory was introduced and a Cauchy formula is discussed. Moreover, a detailed study of the zeros is performed. Finally, we find a formula for the determinant of a matrix with entries in $ \mathcal{Q}_{\mathbb{R}_3}$.
\end{abstract}



\noindent {\em }

\tableofcontents
\section{Introduction}
The skew associative and division algebra of quaternions, $\mathbb{R}_2$, was introduced by Sir Hamilton in 1843 on a famous bridge "promenade": at the beginning he simply wanted to extend to dimension 3 the remarkable properties of the complex plane but he didn't succeed and hence he decided to move to dimension 4. He proposed 3 imaginary units $i,j,k$ such that $i^2=j^2=k^2=-1$ and that anti-commute in the sense that for example  $ij=-ji=k.$  It turns out that, as a vector space, $\mathbb{R}_2$ is isomorphic to $\mathbb{R}^4$. Differently from the complex space $\mathbb{C}$, the space of quaternions has infinity many imaginary units which compose a 2-real dimensional sphere $\mathbb{S} = \{ I=x_1i+x_2j+x_3k \,\, | \,\,  I^2=-1 \}$.
 In view of this remark, it is possible to give to quaternions a so called "book structure" :
$$
\mathbb{R}_2 = \cup_{I \in \mathbb{S}} \,\,\,  \mathbb{C}_I
$$
where $\mathbb{C}_I= x+yI$ ($x,y \in \mathbb{R}$ and $I \in \mathbb{S}$)  is a  complex plane with imaginary unit $I$ and passing through the real line $\Re e q =x_0=0$. \\
Over $\mathbb{C}$ the three following facts are equivalent : \\
\begin{itemize}
\item[1)] to admit a complex first derivative \\
\item[2)] to satisfy Cauchy-Riemann equations \\
\item[3)] to admit a power series expansion. \\
\end{itemize}
This is not the case in the quaternionic setting : indeed  if you consider quaternionic maps such that the limit for $h \to 0$ of $h^{-1}(f(q+h)-f(q))$ exists and it is finite, you find only quaternionic affine maps $f(q)= a+qb$, with $h,a, b, q \in \mathbb{R}_2.$ \\
In 1935 Fueter proposed the following definition of "regular" function over $\mathbb{R}_2$: a function $f$ is "regular" in the sense of Fueter if it is in the kernel of the following operator \\
$$
\frac{\partial }{\partial \overline{q}} = \frac{1}{4} \left(\frac{\partial }{\partial x_0} + \frac{\partial }{\partial x_1} + \frac{\partial }{\partial x_2} + \frac{\partial }{\partial x_3}\right).
$$
All harmonic functions over $\mathbb{R}^4$ are Fueter regular, but $f(q)=q$, the identity map over $\mathbb{R}_2$, is not, which is quite a big inconvenient.
\\ On the other hand, if you consider quaternionic functions that admit the following power series expansion:
$$
f(q)= \sum_{n \in \mathbb{N}} P_n (q-p)
$$
where $P_n(q)=\sum_{l=0}^n  a_0qa_1qa_2...qa_l$ you obtain the class of all the real analytic functions, hence nothng new. \\
Finally if you consider the functions $f$ in the  kernel of the following third order operator : \\
$$
\frac{\partial }{\partial \overline{q}} \Delta
$$
you obtain the class of the so called {\it quaternionic holomorphic} functions which of course contain the harmonic ones on $\mathbb{R}^4$. \\
But very promising was the Cullen operator introduced firstly in 1965 by Cullen and then reformulated in a more geometric way in 2006-2007 by Gentili and Struppa, which was the starting point of a very rich and fruitful theory, the which one of Slice Regular functions, a very active research area nowadays. \\
The Cullen operator is the following :
$$
\left(\frac{\partial}{\partial x_0} +\frac{\Im m q}{r} \frac{\partial}{\partial r}\right)
$$
where $\Re e q =x_0$, $\Im m q=x_1 i + x_2 j + x_3 k$ and $r =\sqrt{x_1^2+x_2^2+x_3^2}$. \\
The definition of slice regular function proposed by Gentili and Struppa is the following : \\
Let $f : \Omega \to \mathbb{R}_2$ be  a real differentiable map from an open set $\Omega \subset \mathbb{R}_2$ such that $\Omega \cap \mathbb{R} \neq \emptyset.$
$f$ is {\it Slice Regular} if
$$
\left(\frac{\partial}{\partial x} +I \frac{\partial}{\partial y}\right)f=0
$$
A large class of quaternionic maps are slice regular: polynomials with all coefficients to the right, any power series expansions centered in 0, the "exponential" function etc... \\
Many results in the flavour of complex analysis still hold true, if properly stated,  in the setting of slice regular functions; only to cite  some of them : the Cauchy Formula, the Cauchy estimates, the Maximum Modulus  Principle, the Liouville and Morera theorems, the Identity Principle, the Open mapping theorem. \\
In this paper, the authors enter into the details of slice regularity of maps defined on $\mathbb{R}_3$ via the important isomorphism of $\mathbb{R}$-algebras
$$
\mathbb{R}_3=\omega_+ \mathbb{R}_3^+ \oplus \omega_{-}\mathbb{R}_3^{+} \cong \mathbb{R}_2 \oplus \mathbb{R}_2, \qquad \omega_{\pm}= \frac{1}{2}(e_{0} \pm e_{123}),
$$
and via a clever remark on the structure of the quadratic cone of this particular algebra $\mathbb{R}_3$. \\ Indeed in Section 3 we prove that :
$$
\mathcal{Q}_{\mathbb{R}_3} \simeq \{ (x+yI,x+yJ) \,\, | \,\, I, J \in \mathbb{S}, (x,y) \in \mathbb{R}^2 \}.
$$
Thanks to this slice structure of the cone of this special algebra, it is possible to understand better slice regularity for functions defined on $\mathcal{Q}_{\mathbb{R}_3}$ that we have called  bi-slice regularity, see  Section 4, where the following differential operator is introduced
\begin{equation}
\overline{\partial}_{IJ}:= \frac{1}{2}\left(\omega_{+}(\partial_x+I \partial_y)+ \omega_{-}(\partial_x+J \partial_y) \right)
, \qquad \forall I \in \mathbb{S}_{\mathbb{R}_2}, \, \, \forall J \in \mathbb{S}_{\mathbb{R}_2},
\end{equation}
A remarkable and useful fact is that a bi-slice regular function defined on a ball centered at the origin inside $\mathcal{Q}_{\mathbb{R}_3}$ can be expanded as a linear combination of two quaternionic power series:
$$ f(x)=f(\omega_{+}p+ \omega_{-}q)= \omega_{+}\sum_{n=0}^{+\infty} p^n b_{n}+ \omega_{-} \sum_{n=0}^{+\infty} q^n c_{n},$$
where $p,q \in \mathbb{R}_2$, with $ \hbox{Re}(p)= \hbox{Re}(q)$, $ | \hbox{Im}(p)|=|\hbox{Im}(q)|$, and $ \{b_{n} \}_{n \in \mathbb{N}}, \{c_{n} \}_{n \in \mathbb{N}} \subseteq \mathbb{R}_2.$
Moreover, for the bi-slice regular functions hold a representation formula and a slitting lemma.
\\The splitting of the algebra $\mathbb{R}_3$ allows also to write a new Cauchy Formula as a linear combination of two quaternionic Cauchy formulas, see Section 5. 
\\In Section 6 we study the zeros and their multiplicities. Starting by the observation that
the $*-$product of polynomials on $\mathbb{R}_3$ descends on the two copies of $\mathbb{R}_2$, we give a precise description of zeros of bi-slice regular functions taking advantage of what is known over the quaternions. Precisely, we perform a detailed study of the zeros of the following polynomial
$$ p(x)=(x- \alpha)*(x- \beta),$$
where $ \alpha$, $ \beta \in \mathbb{R}_3$. In particular we have corrected an imprecise result of Ghiloni-Perotti-Stopptato, see Example 9.6 in \cite{GPS} and Proposition 4.10 in \cite{GPS1}. Moreover, we suggest a new kind of multiplicity.
\\Finally, as an application, in section 8, we use the special feature of the quadratic cone of $\mathbb{R}_3$ to obtain a nice formula for the determinant of a matrix with coefficients in $\mathcal{Q}_{\mathbb{R}_3} $.

\section{Preliminary material}
In this section we will overview and collect the main notions and results needed for our aims.
First, let us recall that the skew field of quaternions
 may be identified with the Clifford algebra $ \mathbb{R}_2$.
An element $q \in \mathbb{R}_2$ is usually written as $q=x_0+ix_1+jx_2 +kx_3$, where $i^2=j^2=k^2=-1$ and $ijk=-1$. Given a quaternion $q$, we introduce a conjugation in $ \mathbb{R}_2$ (the usual one), as $q^c=x_0-ix_1-jx_2 -kx_3$; with this conjugation we define the real part of $q$ as $Re(q):= (q+q^c)/2$ and the imaginary part of $q$ as $Im(q)=(q-q^c)/2$. With the defined conjugation we can write the euclidian square norm of a quaternion $q$ as $|q|^2=qq^c$. The subalgebra of real numbers will be identified, of course, with the set $ \mathbb{R}= \{ q \in \mathbb{R}_2 \, \,| \, \, Im(q)=0 \}$.

Now, if $q$ is such that $Re(q)=0$, then the imaginary part of $q$ is such that $ (Im(q)/|Im(q)|)^2=-1$. More precisely, any imaginary quaternion $I=ix_1+jx_2 +kx_3$, such that $x^2_1+x^2_2+x^2_3=1$ is an imaginary unit. The set of imaginary units is then a real 2-sphere and it will be conveniently denoted as follows
$$ \mathbb{S}_{\mathbb{R}_2}:= \{ q \in \mathbb{R}_2\,\,| \,\, q^2=-1 \}= \{q  \in \mathbb{R}_2\,\,| \,\, Re(q)=0, \, |q|=1 \}.$$
With the previous notation, any $q \in \mathbb{R}_2$ can be written as $q= \alpha+I \beta$, where $ \alpha, \beta \in \mathbb{R}$ and $ I \in \mathbb{S}_{\mathbb{R}_2}$.
\noindent
Given any $I \in \mathbb{S}_{\mathbb{R}_2}$ we will denote the real subspace of $ \mathbb{R}_2$ generated by 1 and $I$ as
$$ \mathbb{C}_I:= \{ q \in \mathbb{R}_2 \,\, | \,\, q= \alpha+I \beta, \, \, \, \alpha, \beta \in \mathbb{R} \}.$$
Sets of the previous kind will be called \emph{slices}
and they are also complex planes with respect to the complex
  structure defined by the respective parameter $I$.
All these notations reveal now clearly the \emph{slice} structure of $ \mathbb{R}_2$ as a union of complex planes $ \mathbb{C}_I$ for $I$ which varies in $ \mathbb{S}_{\mathbb{R}_2}$, i.e.
$$ \mathbb{R}_2= \bigcup_{I \in \mathbb{S}_{\mathbb{R}_2}}\mathbb{C}_{I}, \quad \bigcap_{I \in \mathbb{S}_{\mathbb{R}_2}} \mathbb{C}_I= \mathbb{R}.$$
We denote the real 2-sphere with center $ \alpha \in \mathbb{R}$ and radius $| \beta|$ with $ \beta \in \mathbb{R}$ (passing through $ \alpha+I\beta \in \mathbb{R}_2$), as:
\begin{equation}
\label{rr1}
\mathbb{S}^{\alpha+I \beta}_{\mathbb{R}_2}:= \{ q \in \mathbb{R}_2 \,\, |\,\,  q= \alpha+I \beta, \,\,\, I \in \mathbb{S}_{\mathbb{R}_2} \} \subset \mathbb{R}_2.
\end{equation}
Obviously, if $\beta=0$, then $ \mathbb{S}^{\alpha}_{\mathbb{R}_2}= \{ \alpha \}$.

The following notion of slice regularity was introduced by Gentili and Struppa \cite{GS,GSS}.

\begin{Def}
Let $ \Omega$ be an open subset of $ \mathbb{R}_2$ with $\Omega \cap \mathbb{R} \neq \emptyset$. A real differentiable function $f: \Omega \mapsto \mathbb{R}_2$ is slice regular if for every $I \in \mathbb{S}_{\mathbb{R}_2}$ its restriction $f_I$ to the complex plane $\mathbb{C}_I$ passing through the origin and containing $1$ and $I$ is holomorphic on $\Omega \cap \mathbb{C}_I$.
\end{Def}

For a ball in $\mathbb{R}_2$ centered at the origin we have that a slice regular function can be represented by a convergent power series
$$ f(q)= \sum_{k=0}^{+ \infty} q^k a_k, \qquad \{a_k\}_{k \in \mathbb{N}} \subset \mathbb{R}_2.$$

The theory of slice regular functions has given already many fruitful results, both on the analytic and the geometric side, see for example \cite{ANB, AB, BAdM, BG1, BG2, BG3, BS1, BS2, BS3, BW, BW2}.

Now we will see some basic notions about the real Clifford algebra $ \mathbb{R}_3$ and its quadratic cone $\mathcal{Q}_{\mathbb{R}_3}$, introduced in the papers \cite{GP, GP1}.
\\ We define $\mathbb{R}_3$ as the associative non-commutative algebra defined as follows. Let $ \{ e_1, e_2, e_3 \}$ be the canonical orthonormal basis for $ \mathbb{R}^3$.
Then $\mathbb{R}_3$ is the real associative algebra with $1$ generated
  by the $e_i$
with defining relations $ e_{i}e_{j}+e_{j}e_{i}=-2 \delta_{ij}$.
  This is the real Clifford algebra for the vector space $\mathbb{R}^3$
  with the standard euclidean  quadratic form.
In the sequel we will write $e_0:=1$, $e_{i} e_{j}= e_{ij}$, for $i,j=1,2,3$, $i \neq j$, and $e_{1}e_{2}e_{3}=e_{123}.$ Thus an arbitrary element $x \in \mathbb{R}_3$ can be written as
\begin{equation}\label{class}
 x=x_0e_0+x_1e_1+x_2e_2+x_3e_3+x_{12}e_{12}+x_{13}e_{13}+x_{23}e_{23}+x_{123}e_{123}
\end{equation}
where the coefficients $x_i$, $x_{ij}$, $x_{ijk}$ are real numbers. Thus, we see that $\mathbb{R}_3$ is an eight dimensional real space, endowed with a natural multiplicative structure.
The conjugate of $x$ will be denoted by $\bar{x}$
and can be defined
as the unique antivolution%
\footnote{An antinvolution
  is a linear self-map of order $2$ such that
  $\overline{xy}=(\bar y)\cdot (\bar x)\ \forall x,y \in A$, with $A$ any real quadratic alternative algebra with a unity}
 of ${\mathbb R}_3$
with $e_i\mapsto \bar{e_i}=-e_i$.
Conjugation may likewise be defined extending by linearity the anti-involution
$$\overline{e_0}=e_0, \quad \overline{e_i}=-e_i, \quad \overline{e_{ij}}=-e_{ij}, \quad  \overline{e_{123}}= e_{123},$$
for $i,j \in \{1,2,3 \}, \, i \neq j. $

Moreover, it is known that in $ \mathbb{R}_3$ one can consider the two idempotents $ \omega_{+}= \frac{1}{2}(e_0+e_{123})$ and $ \omega_{-}= \frac{1}{2}(e_0-e_{123})$ (i.e. $\omega^2_{+}=\omega_{+}$ , $\omega^2_{-}=\omega_{-}$), that are mutually annihilating each other i.e. $ \omega_{+} \omega_{-}=\omega_{-} \omega_{+}=0$ (see \cite[Chapter 6]{CSS2}, \cite{DS}).
  Let $ \mathbb{R}_3^{+}$ denote the even subalgebra of $ \mathbb{R}_3$
  i.e.
$$ \mathbb{R}_3^{+}= \{ x_{0}e_{0}+x_{23}e_{23}+ x_{12} e_{12}+x_{13}e_{13}: \, x_{0},x_{23},x_{12}, x_{12} \in \mathbb{R} \}.$$
Note that $\mathbb{R}_3^{+}\simeq\mathbb{R}_2$ as $\mathbb{R}$-algebras.

Every $x \in \mathbb{R}_3$ admits a unique representation
\begin{equation}\label{split}
x=\omega_+p+\omega_-q
\end{equation}
with $q,p \in \mathbb{R}_{3}^{+} \simeq \mathbb{R}_2$.
So we have the isomorphism of $\mathbb{R}$-algebras \cite{DSS}
\begin{equation}
\label{ssplit}
\mathbb{R}_3 =\omega_+\mathbb{R}_3^{+}
  \oplus\omega_-\mathbb{R}_3^{+}\simeq \mathbb{R}_2 \oplus \mathbb{R}_2
\end{equation}
where the ring structure is given by $(p,q)+(p',q')=(p+p', q'+q')$ and $(p,q)(p', q')=(pp',qq')$.
The equality \eqref{split} is very useful since it helps us to work in the Clifford algebra $ \mathbb{R}_3$ using the quaternionic results.

Conjugation on $ \mathbb{R}_3$ is compatible with conjugation
on $\mathbb{R}_2$ via this splitting:
\[
\bar x=\omega_{+} p^c+ \omega_{-} q^c, \qquad \overline{\omega_{+}}=\omega_{+}, \quad \overline{\omega_{-}}=\omega_{-},
\]
for $x=\omega_+p+\omega_-q$
where $p^c$ and $q^c$ are conjugate of $p$
and $q$ as elements of $\mathbb{R}_2$.

As usual, we have the notions of {\em norm} and {\em trace} associated
to the conjugation, i.e., the norm $n(x)$ is defined as $x\bar x$ and
the trace $t(x)$ is defined as $t(x)=x+\bar x$.

With respect to the splitting
$x=\omega_+p+\omega_-q$ we obtain:
\begin{equation}
\label{split-formula}
t(x)=\omega_+t(p)+\omega_-t(q),\quad
n(x)=\omega_+n(p)+\omega_-n(q)
\end{equation}

The norm is multiplicative, i.e., $n(xy)=n(x)n(y)$, $ \forall x,y\in
{\mathbb R}_3$.
\newline
\newline
\textbf{Notation}
In the sequel we will denote the element $ \omega_{+}p+ \omega_{-}q$ as the couple $(p,q)$.
\newline
\newline
For more details about this splitting the interested reader can see \cite{DSS, R1,R,SS}.
\begin{nb}
  In general, it is known that every Clifford algebra ${\mathbb R}_n$
  is either a matrix algebra of rank $r\ge 1$ over
  $ \mathbb{R}$, $ \mathbb{C}$ or $ \mathbb{R}_2$
  or a direct sum of two copies of such a matrix algebra
  \cite{LM, P}.

  An explicit proof of the splitting
  \[
  \mathbb{R}_3\simeq Mat(1\times 1,\mathbb{R}_2)\oplus Mat(1\times 1,\mathbb{R}_2)
  \simeq \mathbb{R}_2\oplus \mathbb{R}_2
  \]
  may be found in the papers \cite{R1,R}. However, the reader should be
  aware that
  these papers also contain an incorrect claim that $ \mathbb{R}_n$ is isomorphic to a sum of $2^{n-1}$ copies of the algebra $ \mathbb{R}_2$.
  \end{nb}
Now, we introduce some basic facts about the quadratic cone \cite{GP, GP1}.
\begin{Def}[\cite{GP}]
We call quadratic cone of $ \mathbb{R}_3$ the set
$$ \mathcal{Q}_{\mathbb{R}_3}:= \mathbb{R} \cup \{ x \in \mathbb{R}_3 \setminus \mathbb{R} \, \, | \, t(x) \in \mathbb{R}, \, n(x) \in \mathbb{R}, \, 4n(x)> t(x)^2 \}.$$
\end{Def}

We also define $ \mathbb{S}_{\mathbb{R}_3}:= \{x \in \mathcal{Q}_{\mathbb{R}_3}| \, x^2=-1 \}. $ The elements of $ \mathbb{S}_{\mathbb{R}_3}$ will be called \emph{square roots} of $-1$ in the algebra $ \mathbb{R}_3$.

It is has been proven in \cite[Prop. 2.1]{BDW} the following result
\begin{prop}
\label{split2}
$$
\mathbb{S}_{\mathbb{R}_3}
=\left\{ \omega_+p+\omega_-q: p,q\in {\mathbb R}_2, p^2=q^2=-1
\right\} \simeq \omega_{+}\mathbb{S}_{\mathbb{R}_2} \oplus \omega_{-}\mathbb{S}_{\mathbb{R}_2}
$$
\end{prop}

The following result, proved in \cite[Prop. 3]{GP}, will be important for our results.
\begin{prop}
\label{one}
The following statements hold
\begin{enumerate}
\item $\mathcal{Q}_{\mathbb{R}_3}= \bigcup_{J \in \mathbb{S}_{\mathbb{R}_3}} \mathbb{C}_J,$
\item If $I,J \in \mathbb{S}_{\mathbb{R}_3}$, $I \neq \pm J$, then $ \mathbb{C}_I \cap \mathbb{C}_J= \mathbb{R}$.
\end{enumerate}
\end{prop}
In \cite{GP} the authors studied the quadratic cone for a general finite-dimensional real alternative algebra with unity. They remark that if the algebra is isomorphic to one of the division algebras $\mathbb{R}_2$, $\mathbb{O} $  we have that $ \mathcal{Q}_{\mathbb{R}_2}= \mathbb{R}_2$ and $ \mathcal{Q}_\mathbb{O}= \mathbb{O}$. Furthermore, in these cases, $\mathbb{S}_{\mathbb{R}_2}= \{ q \in \mathbb{R}_2 : \, q^2 =-1 \}$, i.e. is a 2-sphere, and $ \mathbb{S}_{\mathbb{O}}$ is a 6-sphere.

In the case of the Clifford algebra $ \mathbb{R}_n$, for $ n\geq 3$, the quadratic cone $ \mathcal{Q}_{\mathbb{R}_n}$ is a real algebraic subset of $ \mathbb{R}_n$. Now, we recall that an element $(x_0, x_1,...,x_n) \in \mathbb{R}^{n+1}$ can be identified with the element
$$ x=x_0 e_0+ \sum_{j=1}^{n} x_je_j,$$
called, in short paravector. We remark that the subspaces of paravectors $ \mathbb{R}^{n+1}$ is contained in $\mathcal{Q}_{\mathbb{R}_n}$. Moreover, the $(n-1)$- dimensional sphere $ \mathbb{S}= \{x=x_{1}e_1+...+x_ne_n \in \mathbb{R}^{n}| \, x_1^2+...+x_n^2=1 \}$ of unit imaginary vectors is properly contained in $ \mathbb{S}_{\mathbb{R}_n}$.

In particular, for the case $ \mathbb{R}_3$ it is possible to show that the quadratic cone is the 6-dimensional real algebraic set
\begin{equation}
\label{fqc}
\mathcal{Q}_{\mathbb{R}_3}= \{x \in \mathbb{R}_3 \,|  \,x_{123}=0, \, x_2x_{13}-x_1x_{23}-x_3 x_{12}=0\}.
\end{equation}

\begin{Def}
Let $U \subseteq \mathcal{Q}_{\mathbb{R}_n}$ be a domain. We say that $U$ is a slice domain (s-domain for short) if $U \cap \mathbb{R}$ is non-empty and if $U \cap \mathbb{C}_I$ is a domain in $ \mathbb{C}_I$ for all $I \in \mathbb{S}_{\mathbb{R}_n}$.
\end{Def}
 \begin{Def}
 \label{ax}
Let us consider an open subset $D$ of $ \mathbb{C}$, we define $ \Omega_D^4$ as a subset of $\mathbb{R}_3$ such that
$$ \Omega_D^4:= \{x= \alpha+\beta J \in \mathbb{C}_J|\,\alpha, \beta \in \mathbb{R}, \, \alpha+i \beta \in D, \, J \in \mathbb{S}_{\mathbb{R}_3} \}.$$
This kind of set will be called axially symmetric domain. Furthermore, we set
\begin{equation}
\label{ax1}
\Omega_D^2:=\{x= \alpha+\beta K \in \mathbb{C}_K|\,\alpha, \beta \in \mathbb{R}, \, \alpha+i \beta \in D, \, K \in \mathbb{S}_{\mathbb{R}_2} \}.
\end{equation}
We call this kind of set quaternionic axially symmetric domain.
\end{Def}
We define $ \mathbb{R}_3 \otimes \mathbb{C}$ as the complexification of $ \mathbb{R}_3$. We will use the representation
$$ \mathbb{R}_3 \otimes \mathbb{C}= \{ w=x+\iota y|\, x,y \in \mathbb{R}_3 \} \qquad (\iota^2=-1).$$
 \begin{Def}
 \label{slice}
A function $F: D \subseteq \mathbb{C} \to \mathbb{R}_3 \otimes \mathbb{C}$ defined by $F(z)= F_1(z)+i F_2(z)$ where $z= \alpha+i \beta \in D$, $F_1,F_2: D \to \mathbb{R}_3$, and where $F_1(\bar{z})=F_1(z)$ and $F_2(\bar{z})=-F_2(z)$ whenever $z, \bar{z} \in D$, is called stem function. Given a stem function $F$, the function $f= \mathcal{I}(F)$ defined by
 $$ f(x)=f(\alpha+ \beta J):=F_1(z)+JF_2(z)$$
 for any $ x \in \Omega_D^4$ is called the slice function induced by $F$.
\end{Def}
\begin{Def}
\label{SR}
A (left) slice function $f \in \mathcal{S}(\Omega_D^4)$ is (left) slice regular if its stem function $F$ is holomorphic.
\end{Def}

\section{The quadratic cone of $ \mathbb{R}_3$}
In this section we show the peculiar structure of the quadratic cone of $ \mathbb{R}_3$.
\begin{theorem}
\label{qc}
It is possible to write the quadratic cone of $ \mathbb{R}_3$ as a couple of quaternions with the same real part and the same modulus of the imaginary part, i.e.
\begin{equation}
\mathcal{Q}_{\mathbb{R}_3}\simeq \left \{(x+yI, x+yJ) \, | \, \, I,J \in \mathbb{S}_{\mathbb{R}_2}\, , \, \, (x,y) \in \mathbb{R}^2 \right\}.
\end{equation}
\end{theorem}
\begin{proof}
First of all, we show how to pass from the eight coordinates of $ \mathbb{R}_3$ to the two quadruples of quaternionic coordinates. Let us consider a generic element $x \in \mathbb{R}_3$. By \eqref{split} we have
\begin{equation}
\label{s1}
x= \omega_{+} p+ \omega_{-} q,
\end{equation}
where $p,q \in \mathbb{R}_3^{+}$. This means that
$$ p= p_0e_0+p_{23}e_{23}+p_{12}e_{12}+p_{13}e_{13},$$
$$ q= q_0e_0+q_{23}e_{23}+q_{12}e_{12}+q_{13}e_{13}.$$
where $p_0, p_{23}, p_{12}, p_{13}, q_0, q_{23}, q_{12}, q_{13} \in \mathbb{R}$. Now, we substitute in \eqref{s1}
\begin{eqnarray*}
&& x_0e_0+x_1e_1+x_2e_2+x_3e_3+x_{12}e_{12}+x_{13}e_{13}+x_{23}e_{23}+x_{123}e_{123}\\
&&=\omega_{+}(p_0e_0+p_{23}e_{23}+p_{12}e_{12}+p_{13}e_{13}) + \omega_{-}(q_0e_0+q_{23}e_{23}+q_{12}e_{12}+q_{13}e_{13}).
\end{eqnarray*}
Recalling the definition of $ \omega_{+}$ and $ \omega_{-}$ we get
\begin{eqnarray*}
&&\! \! \! \! \! x_0e_0+x_1e_1+x_2e_2+x_3e_3+x_{12}e_{12}+x_{13}e_{13}+x_{23}e_{23}+x_{123}e_{123}\\
&& \! \! \! \! \!= \frac{1}{2}(p_0+q_0)e_0+ \frac{1}{2}(p_{23}+q_{23})e_{23}+\frac{1}{2}(p_{12}+q_{12})e_{12}+\frac{1}{2}(p_{13}+q_{13})e_{13}+\\
&&+ \frac{1}{2}(p_0-q_0) e_{123}+ \frac{1}{2}(-p_{12}+q_{12}) e_3+ \frac{1}{2}(-p_{23}+q_{23})e_1+ \frac{1}{2}(p_{13}-q_{13})e_2.
\end{eqnarray*}
This implies the following linear system
\begin{equation}
\label{sys}
\begin{cases}
\frac{1}{2}(p_0+q_0)=x_0\\
\frac{1}{2}(p_{23}+q_{23})=x_{23}\\
\frac{1}{2}(p_{12}+q_{12})=x_{12}\\
\frac{1}{2}(p_{13}+q_{13})=x_{13}\\
\frac{1}{2}(p_0-q_0)=x_{123}\\
\frac{1}{2}(-p_{12}+q_{12})=x_3\\
\frac{1}{2}(-p_{23}+q_{23})=x_1\\
\frac{1}{2}(p_{13}-q_{13})=x_2.
\end{cases}
\end{equation}
From \eqref{fqc} we know the equations of $ \mathcal{Q}_{\mathbb{R}_3}$ in the variables $x_{i}$, $x_{j}$ and $x_{jk}$. Using the linear system \eqref{sys} we can write the quadratic cone of $ \mathbb{R}_3$ in quaternionic coordinates.  Therefore by the first equation of the quadratic cone, $x_{123}=0$ we get
\begin{equation}
\label{fe}
p_{0}=q_{0}.
\end{equation}
From the second equation, $x_2x_{13}-x_{1}x_{23}-x_3 x_{12}=0$, we get
\begin{eqnarray}
\frac{1}{4}(p_{13}-q_{13})(q_{13}+p_{13})-\frac{1}{4}(p_{23}-q_{23})(p_{23}+q_{23})-\frac{1}{4}(p_{12}-q_{12})(p_{12}+q_{12})=0.
\end{eqnarray}
Hence
\begin{equation}
\label{se}
p_{13}^2+p_{23}^2+p_{12}^2=q_{13}^2+q_{23}^2+q_{12}^2.
\end{equation}
Finally, from \eqref{fe} and \eqref{se} we get that the two quaternions $p,q$ stay in the quadratic cone if they have the same real part and the same modulus of the imaginary part.
\end{proof}

\begin{nb}
It is also possible to show the previous result by combining Proposition \ref{split2} and Proposition \ref{one}. Indeed, if $z \in \mathcal{Q}_{\mathbb{R}_3}$ then
$$ z \in \bigcup_{J \in \mathbb{S}_{\mathbb{R}_3}} \mathbb{C}_J.$$
So there exists $K \in \mathbb{S}_{\mathbb{R}_3}$ such that
$$ z= x+Ky.$$
By Proposition \ref{split2} we have that
$$ K= \omega_{+}I+ \omega_{-}J, \qquad I,J \in \mathbb{S}^2.$$
Then
\begin{eqnarray*}
z&=&x+(\omega_{+}I+ \omega_{-}J)y\\
&=& (\omega_{+}+ \omega_-)x+(\omega_{+}I+ \omega_{-}J)y\\
&=& \omega_{+}(x+Iy)+\omega_{-}(x+Jy)\\
&:=& \omega_{+} p + \omega_{-} q,
\end{eqnarray*}
 where $p, q \in \mathbb{R}_2$ such that $ \hbox{Re}(p)=\hbox{Re}(q)$ and $ | \hbox{Im(p)}|=|\hbox{Im(q)}|.$
\end{nb}

\begin{nb}
\label{norm}
Since the quadratic cone of $ \mathcal{Q}_{\mathbb{R}_3}$ is a couple of quaternions with the same real part and the same modulus of the imaginary part, we have that
$$ x= \omega_{+} p+ \omega_{-} q \in \mathcal{Q}_{\mathbb{R}_3} \Longleftrightarrow \begin{cases}
t(p)=t(q)\\
n(p)=n(q)
\end{cases},$$
where $p,q \in \mathbb{R}_2$. This implies that
$$ t(x)= \omega_{+} t(p)+ \omega_{-} t(q)=(\omega_{+}+\omega_{-}) t(q) =t(q)\in \mathbb{R},$$
and
$$ n(x)= \omega_{+} n(p)+ \omega_{-} n(q)= (\omega_{+}+ \omega_{-})n(q)=n(q) \in \mathbb{R}.$$
\end{nb}

The peculiar structure of the quadratic cone of $ \mathcal{Q}_{\mathbb{R}_3}$ suggests the following:

\begin{lemma}
It is possible to write the quadratic cone of $ \mathbb{R}_3$ in the following way
$$ \mathcal{Q}_{\mathbb{R}_3} \simeq \bigcup_{I \in \mathbb{S}_{\mathbb{R}_2}} \mathbb{C}_I \oplus \bigcup_{J \in \mathbb{S}_{\mathbb{R}_2}} \mathbb{C}_J.$$
\end{lemma}
\begin{proof}
The result follows by Proposition \ref{one} and Theorem \ref{qc}.
\end{proof}
\begin{Def}
\label{bix}
Let $D$ be an open subset of $ \mathbb{C}$, let $ \Omega_D^4$ be a subset of $ \mathbb{R}_3$ defined as
$$ \Omega_D^4 \simeq \{(x+yI, x+yJ) \,| \,  x,y \in \mathbb{R} \, , \, \, x+i y \in D, \, \, I,J \in \mathbb{S}_{\mathbb{R}_2}\}.$$
We will call these types of sets bi-axially symmetric domains.
\end{Def}

\begin{nb}
Due to Proposition \ref{split2} the bi-axially symmetric domain is equivalent to the first set defined in Definition \ref{ax}.
\end{nb}

\begin{lemma}
A bi-axially symmetric domain is isomorphic to a direct sum of quaternionic axially symmetric domains.
\end{lemma}
\begin{proof}
This result follows from the Definition of bi-axially symmetric set and the formula \eqref{ax1}
\begin{eqnarray*}
\Omega_D^4&\simeq& \{(x+yI, x+yJ) \,| \,  x,y \in \mathbb{R} \, , \, \, x+i \beta \in D, \, \, I,J \in \mathbb{S}_{\mathbb{R}_2}\}\\
&\simeq& \{x+yI  \,| \,  x,y \in \mathbb{R} \, , \, \, x+i \beta \in D, \, \, I\in \mathbb{S}_{\mathbb{R}_2}\} \oplus \\
&&\oplus \{ x+yJ  \,| \,  x,y \in \mathbb{R} \, , \, \, x+i \beta \in D, \, \, J \in \mathbb{S}_{\mathbb{R}_2}\}\\
&\simeq& \Omega_D^2 \oplus \Omega_D^2.
\end{eqnarray*}
\end{proof}

\subsection{Stem functions}
The peculiar structure of $ \mathcal{Q}_{\mathbb{R}_3}$ and the splitting of $ \mathbb{R}_3$ have an influence on the stem functions.
\begin{prop}
A function $F:D \subseteq \mathbb{C} \to \omega_{+}\left(\mathbb{R}_2 \otimes \mathbb{C}\right) \oplus \omega_{-}\left(\mathbb{R}_2 \otimes \mathbb{C}\right)$  defined by
$$
F(z)= \omega_{+} \left(f_1(z)+ i f_2(z)\right)+ \omega_{-}\left(g_1(z)+i g_2(z)\right), \qquad z=\alpha+ i \beta \in D, \quad f_1,f_2,g_1,g_2:D \to \mathbb{R}_2,
$$
is called stem function. Moreover, the functions $f_1, f_2, g_1, g_2$ have to satisfy the following conditions
\begin{equation}
\label{ed}
\begin{cases}
f_1(z)=f_1(\bar{z})\\
f_2(z)=-f_2(\bar{z})\\
g_1(z)=g_1(\bar{z})\\
g_2(z)=-g_2(\bar{z})
\end{cases}
\end{equation}

whenever $z, \bar{z} \in D$.
\end{prop}
\begin{proof}
In Definition \ref{slice} the stem function is defined as a function $F:D \subset \mathbb{C} \to \mathbb{R}_{3} \otimes \mathbb{C}$. Using the splitting \eqref{ssplit} we can rewrite the codomain as
$$ \mathbb{R}_{3} \otimes \mathbb{C} \simeq(\mathbb{R}_2 \oplus \mathbb{R}_2) \otimes \mathbb{C} \simeq \left(\mathbb{R}_2 \otimes \mathbb{C}\right) \oplus \left(\mathbb{R}_2 \otimes \mathbb{C}\right). $$
Furthermore, in Definition \ref{slice} the stem function is defined as $F(z)=F_1(z)+ i F_2(z)$, where $z= \alpha+ i \beta$ and $F_{1},F_2:D \to \mathbb{R}_3$. Now, we split the functions as
\begin{equation}
\label{p1}
F_{1}(z)= \omega_{+} f_1(z)+ \omega_{-} g_1(z),
\end{equation}
\begin{equation}
\label{p2}
F_2(z)= \omega_{+} f_2(z)+ \omega_{-} g_2(z),
\end{equation}
where $f_1,f_2,g_1,g_2:D \to \mathbb{R}_2$. Then, since $\omega_{+}$ and $\omega_{-}$ stay in the center of the algebra $ \mathbb{R}_{3} \otimes \mathbb{C}$ this means that commute with $i$. Therefore,
\begin{eqnarray*}
F(z)&=& F_1(z)+ i F_2(z)= \omega_{+} f_1(z)+ \omega_{-} g_1(z)+ i \left(\omega_{+} f_2(z)+ \omega_{-} g_2(z)\right)\\
&=& \omega_{+} \left(f_1(z)+ i f_2(z)\right)+ \omega_{-}\left(g_1(z)+i g_2(z)\right).
\end{eqnarray*}
As well as functions the $F_1$ and $F_2$ have to satisfy the even-odd conditions, also the functions $f_1,f_2, g_1, g_2$ satisfy something similar. Indeed, if
$$ F( \alpha+i \beta)=F \left(\alpha-(i) (-\beta) \right),$$
then  we have
$$ \omega_{+} \left(f_1(z)+ i f_2(z)\right)+ \omega_{-}\left(g_1(z)+i g_2(z)\right)=\omega_{+} \left(f_1(\bar{z})- i f_2(\bar{z})\right)+ \omega_{-}\left(g_1(\bar{z})-i g_2(\bar{z})\right)$$
This is equivalent to
$$ \omega_{+} [ f_1(z)-f_1(\bar{z})+ i \left(f_2(z)+f_{2}(\bar{z})\right)]+\omega_{-} [g_1(z)-g_1(\bar{z})+ i \left(g_2(z)+g_{2}(\bar{z})\right)]=0.$$
Since $ \omega_{+}$ and $ \omega_{-}$ are orthogonal we get
$$ \begin{cases}
f_1(z)-f_1(\bar{z})+ i (f_2(z)+f_{2}(\bar{z}))=0\\
g_1(z)-g_1(\bar{z})+ i (g_2(z)+g_{2}(\bar{z}))=0.
\end{cases}
$$
Hence
$$ \begin{cases}
f_1(z)=f_1(\bar{z})\\
f_2(z)=-f_2(\bar{z})\\
g_1(z)=g_1(\bar{z})\\
g_2(z)=-g_2(\bar{z}).
\end{cases}
$$
\end{proof}
The previous result suggests the following:
\begin{Def}
Let us consider a stem function $F$, the function $f= \mathcal{I}(F)$ defined by
$$ f(x)=f(\alpha+\beta J)=\omega_{+} \left(f_1(z)+ J f_2(z)\right)+ \omega_{-}\left(g_1(z)+J g_2(z)\right), \quad z= \alpha+i \beta,$$
for any $x \in \Omega_{D}^4$ is called bi-slice function induced by $F$.
\end{Def}

\begin{nb}
The function $f(x)$ is well defined due to conditions \eqref{ed}. Indeed
\begin{eqnarray*}
f(\alpha+(-J) (-\beta))&=& \omega_{+} \left(f_1(\bar{z})-J f_2(\bar{z})\right)+ \omega_{-}\left(g_1(\bar{z})-Jg_2(\bar{z})\right)\\
&=& \omega_{+} \left(f_1(z)+J f_2(z)\right)+ \omega_{-}\left(g_1(z)+Jg_2(z)\right)\\
&=& f( \alpha+J \beta).
\end{eqnarray*}
\end{nb}

\begin{prop}
A stem function $F:D \to \omega_{+}\left(\mathbb{R}_2 \otimes \mathbb{C}\right) \oplus \omega_{-}\left(\mathbb{R}_2 \otimes \mathbb{C}\right)$ is holomorphic if and only if its quaternionic components $f_1, f_2, g_1, g_2$ satisfy the Cauchy-Riemann conditions:
$$ \frac{\partial f_1}{\partial \alpha}=\frac{\partial f_2}{\partial \beta}, \qquad \frac{\partial g_1}{\partial \alpha}=\frac{\partial g_2}{\partial \beta}, \qquad \frac{\partial f_1}{\partial \beta}=-\frac{\partial f_2}{\partial \alpha}, \qquad
\frac{\partial g_1}{\partial \beta}=-\frac{\partial g_2}{\partial \alpha},$$
where $z= \alpha+ i \beta \in D$.
\end{prop}
\begin{proof}
A function $F(z)$, with $z= \alpha+i \beta$ is holomorphic if and only if
\begin{equation}
\label{cr1}
\frac{\partial F_1}{\partial \alpha}= \frac{\partial F_2}{\partial \beta},
\end{equation}
\begin{equation}
\label{cr2}
\frac{\partial F_1}{\partial \beta }= - \frac{\partial F_2}{\partial \alpha}.
\end{equation}
If we split the functions $F_1$ and $F_2$ as \eqref{p1} and \eqref{p2}, we get by \eqref{cr1}
$$\frac{\partial F_1}{\partial \alpha}= \omega_{+} \frac{\partial f_1}{\partial \alpha}+ \omega_{-} \frac{\partial g_1}{ \partial \alpha}= \omega_{+} \frac{\partial f_2}{\partial \beta}+ \omega_{-} \frac{\partial g_2}{\partial \beta}= \frac{\partial F_2}{\partial \beta}.$$
This implies
$$ \frac{\partial f_1}{\partial \alpha}=\frac{\partial f_2}{\partial \beta}, \qquad \frac{\partial g_1}{\partial \alpha}=\frac{\partial g_2}{\partial \beta}.$$
Similarly, condition \eqref{cr2} gives us
$$\frac{\partial F_1}{\partial \beta}= \omega_{+} \frac{\partial f_1}{\partial \beta}+ \omega_{-} \frac{\partial g_1}{ \partial \beta}= -\omega_{+} \frac{\partial f_2}{\partial \alpha}- \omega_{-} \frac{\partial g_2}{\partial \alpha}= -\frac{\partial F_2}{\partial \alpha}.$$
Hence
$$ \frac{\partial f_1}{\partial \beta}=-\frac{\partial f_2}{\partial \alpha}, \qquad
\frac{\partial g_1}{\partial \beta}=-\frac{\partial g_2}{\partial \alpha}.$$
\end{proof}

\begin{nb}
It is possible to write the spherical value and the spherical derivative using the quaternionic components $f_1, f_2, g_1, g_2$ of the steam functions. Indeed, by \cite{GP} and \eqref{p1} we get
\begin{eqnarray*}
\partial_s f(x)&=& \frac{\beta^{-1}}{2} \left( \omega_{+} f_2(z)+ \omega_{-}g_2(z)\right)\\
&=& \omega_{+}  \partial_s (f_1+if_2) + \omega_{-} \partial_s (g_1+ig_2)  , \qquad x= \alpha+J \beta.
\end{eqnarray*}
Similarly, by \cite{GP} and \eqref{p2} we get
\begin{eqnarray*}
v_s f(x) &=& F_{1}(z) =\omega_{+} f_{1}(z)+ \omega_{-} g_1(z)\\
&=& \omega_{+}v_s(f_1+i f_2)+ \omega_{-} v_s(g_1+ig_2), , \qquad x= \alpha+J \beta.
\end{eqnarray*}
\end{nb}

\section{Bi-slice regularity}
Due to the peculiar structure of the quadratic cone of $ \mathbb{R}_{3}$ and the splitting of $ \mathbb{R}_{3}$ it is possible to give a new concept of regularity.

\begin{Def}
Let $D \subset \mathbb{C}$. Let us consider $f: \mathcal{Q}_{\mathbb{R}_3} \to \mathbb{R}_{3}$. A function $f$ is called bi-slice (left)- regular if
\begin{equation}
\label{reg0}
\overline{\partial}_{IJ}f=0, \qquad \forall I \in \mathbb{S}_{\mathbb{R}_2}, \, \, \forall J \in \mathbb{S}_{\mathbb{R}_2},
\end{equation}
where
$$
\overline{\partial}_{IJ}:= \frac{1}{2}\left(\omega_{+}(\partial_x+I \partial_y)+ \omega_{-}(\partial_x+J \partial_y) \right).
$$
\end{Def}
\begin{theorem}
\label{fund1}
A function $f= \omega_{+}F+\omega_{-}G$, which goes from $\mathcal{Q}_{\mathbb{R}_3}$  to $\mathbb{R}_3$ is bi-slice regular if and only if $F$ and $G$ are slice regular.
\end{theorem}
\begin{proof}
We start by supposing that the function $f$ is bi-slice regular, this means
$$ \overline{\partial}_{IJ}f=0.$$
If $x= \omega_{+}p+ \omega_{-}q$ then
$$ f(x)= \omega_{+}F(p)+ \omega_{-}G(q).$$
Therefore
\begin{eqnarray*}
0 &=& \overline{\partial}_{IJ}f= \frac{1}{2}\left(\omega_{+}(\partial_x+I \partial_y)f+ \omega_{-}(\partial_x+J \partial_y)f \right) \\
&=& \frac{1}{2}\left(\omega_{+}^2(\partial_{x}+I \partial_{y})F+\omega_{+}\omega_{-}(\partial_{x}+J \partial_{y})F +\omega_{-}  \omega_{+}(\partial_{x}+I \partial_{y})F+\omega_{-}^2(\partial_{x}+J \partial_{y})G\right)\\
&=& \frac{1}{2}\left(\omega_{+}(\partial_{x}+I \partial_{y})F+\omega_{-}(\partial_{x}+J \partial_{y})G \right).
\end{eqnarray*}
From the orthogonality of $ \omega_{+}$ and $ \omega_{-}$ we get that
$$ \frac{1}{2}(\partial_{x}+I \partial_{y})F=0,$$
$$ \frac{1}{2}(\partial_{x}+J \partial_{y})G=0.$$
This means that $F$ and $G$ are slice regular functions.
\\ Now, we suppose that $F$ and $G$ are slice regular functions. Then
$$ \frac{1}{2}(\partial_{x}+I \partial_{y})F=0,$$
$$\frac{1}{2}(\partial_{x}+J \partial_{y})G=0.$$
Therefore
$$ \overline{\partial}_{IJ}f= \frac{1}{2}\left(\omega_{+}(\partial_x+I \partial_y)F+ \omega_{-}(\partial_x+J \partial_y)G \right) =0.$$
By definition $f$ is bi-slice regular.
\end{proof}
Now, we show that this notion of regularity is equivalent to that one in \cite[Sec. 4]{GP}.
\begin{prop}
\label{SR3}
A function $f: \mathcal{Q}_{\mathbb{R}_3} \to \mathbb{R}_{3}$ is bi-slice regular if and only if it is regular in the sense of Definition \ref{SR}.
\end{prop}
\begin{proof}
In \cite[Prop. 8]{GP} it is proved that a function $f: \mathcal{Q}_{\mathbb{R}_3} \to \mathbb{R}_{3}$ is slice regular in the sense of Definition \ref{SR} if and only if it stays in the kernel of
\begin{equation}
\label{SR1}
\overline{\partial}_K:= \frac{1}{2} \left( \partial_{x}+K \partial_{y} \right), \qquad \forall K \in \mathbb{S}_{\mathbb{R}_3}.
\end{equation}
By Proposition \ref{split2} we can write
$$ K= \omega_{+}I+ \omega_{-}J, \qquad I,J \in \mathbb{S}_{\mathbb{R}_2}.$$
Moreover, we know that
$$ f(x)= \omega_{+}F(p)+ \omega_{-}G(q).$$
If we suppose that a function $f$ stays in the kernel of the operator defined in \ref{SR1} we get
\begin{eqnarray}
0 &=& \overline{\partial}_K f =\frac{1}{2} \left( \partial_{x}+K \partial_{y} \right)(\omega_{+}F+ \omega_{-}G)\\
&=& \frac{\omega_{+}}{2} \partial_{x}F+\frac{\omega_{-}}{2} \partial_{x}G +K\frac{\omega_{+}}{2} \partial_{y}F+K\frac{\omega_{-}}{2} \partial_{y}G \\
&=& \frac{\omega_{+}}{2} \partial_{x}F+\frac{\omega_{-}}{2} \partial_{x}G +I\frac{\omega_{+}}{2} \partial_{y}F+J\frac{\omega_{-}}{2} \partial_{y}G \\
&=& \frac{\omega_{+}}{2}(\partial_x+I \partial_y)F+ \frac{\omega_{-}}{2}(\partial_x+J \partial_y)G.
\end{eqnarray}
We get
$$ \overline{\partial}_{IJ}f=0.$$
By definition $f$ is bi-slice regular.
\\ On the contrary if we suppose that $f$ is bi-slice regular we obtain
\begin{eqnarray*}
0 &=& \frac{\omega_{+}}{2}(\partial_x+I \partial_y)F+ \frac{\omega_{-}}{2}(\partial_x+J \partial_y)G \\
&=& \frac{1}{2} \partial_{x}(\omega_{+}F+ \omega_{-}{G})+ \frac{1}{2} \partial_{y} \left(I \omega_{+}F+J \omega_{-}G \right) \\
&=& \frac{1}{2} \partial_{x}(\omega_{+}F+ \omega_{-} {G})+ \frac{K}{2} \partial_{y}  \left(\omega_{+}F+\omega_{-}G \right)\\
&=& \frac{1}{2} \left( \partial_{x}+K \partial_{y} \right)f.
\end{eqnarray*}
Hence
$$ \frac{1}{2} \left( \partial_{x}+K \partial_{y} \right)f=0, \quad \forall K \in \mathbb{S}_{\mathbb{R}_{3}}.$$
\end{proof}
For our future computations we need the following result

\begin{prop}
\label{IF}
For all $n \in \mathbb{N}$ and $p,q \in \mathbb{R}_2$ we have
$$ (\omega_{+}p+\omega_{-}q)^n=\omega_{+} p^n+ \omega_{-} q^n.$$
\end{prop}
\begin{proof}
We prove the result by induction on $n$. If $n=2$, by the orthogonality of $ \omega_{+}$ and $ \omega_{-}$ we get
$$ (\omega_{+}p+\omega_{-}q)^2=(\omega_{+}p+\omega_{-}q)(\omega_{+}p+\omega_{-}q)=\omega_{+}p^2+\omega_{-}q^2.$$
Now, we suppose that the statement holds for $n$ and we prove for $n+1$
\begin{eqnarray*}
(\omega_{+}p+\omega_{-}q)^{n+1}&=&(\omega_{+}p+\omega_{-}q)^{n}(\omega_{+}p+\omega_{-}q)\\&=&(\omega_{+}p^n+\omega_{-}q^n)(\omega_{+}p+\omega_{-}q)\\
&=& \omega_{+}^2p^{n+1}+\omega_{-} \omega_{+}q^n p+ \omega_{+} \omega_{-} p^n  q+ \omega_{-}^2 q^{n+1}\\
&=& \omega_{+} p^{n+1}+ \omega_{-} q^{n+1}.
\end{eqnarray*}
\end{proof}
Let us define the ball of the quadratic cone $ \mathcal{Q}_{\mathbb{R}_3}$. Let $R>0$ then
$$ \mathcal{B}(0,R):= \{x \in \mathcal{Q}_{\mathbb{R}_3}: \, \, n(x) <R\}.$$
It is possible to prove that $ \mathcal{B}$ splits in two quaternionic balls.
\begin{lemma}
\label{bs}
Let us consider $x=\omega_{+}p+ \omega_{-}q \in \mathcal{Q}_{\mathbb{R}_3}$, with $p=u+Iv$ and $q=u+Jv$. Then for $R>0$ we have
$$  \mathcal{B}(0,R)=\{p \in \mathbb{R}_2: \, \, n(p)<R\} \oplus \{q \in \mathbb{R}_2: \, \, n(q)<R \}:= B^{\mathbb{R}_2}_p(0,R) \oplus B^{\mathbb{R}_2}_q(0,R).$$
\end{lemma}
\begin{proof}
We prove the equality by double inclusion. Let $x \in \mathcal{B}(0,R)$, then
\begin{equation}
\label{eb}
n(x)= \omega_{+}n(p)+ \omega_{-}n(q)<R
\end{equation}
By Remark \ref{norm} we have that $n(p)=n(q)$. Therefore by \eqref{eb} we get that $n(p)<R$ and $n(q)<R.$
\\ The other inclusion is trivial since $ \omega_{+}+\omega_{-}=1$ and $n(p)<R$, $n(q)<R$ imply $n(x)<R$.
\end{proof}
\begin{theorem}
Let $f: \mathcal{B}(0,R) \subseteq \mathcal{Q}_{\mathbb{R}_3} \to \mathbb{R}_3$ be a bi-slice regular function. Then
$$ f(x)=f(\omega_{+}p+ \omega_{-}q)= \omega_{+}\sum_{n=0}^{+\infty} p^n b_{n}+ \omega_{-} \sum_{n=0}^{+\infty} q^n c_{n},$$
where $p,q \in \mathbb{R}_2$, with $ \hbox{Re}(p)= \hbox{Re}(q)$, $ | \hbox{Im}(p)|=|\hbox{Im}(q)|$, and $ \{b_{n} \}_{n \in \mathbb{N}}, \{c_{n} \}_{n \in \mathbb{N}} \subseteq \mathbb{R}_2.$
\end{theorem}
\begin{proof}
By Proposition \ref{SR3} we know that a function is bi-slice regular if and only if is slice regular in the sense of \ref{SR}. This implies that we can write
$$ f(x)= \sum_{n=0}^{+\infty} x^n a_{n}, \qquad  \{a_{n} \}_{n \in \mathbb{N}} \subseteq{\mathbb{R}_3}.$$
By the splitting of $\mathbb{R}_{3}$ we get that $ a_{n}= \omega_{+} b_{n}+ \omega_{-}c_{n}$, $ \{b_{n} \}_{n \in \mathbb{N}}, \{c_{n} \}_{n \in \mathbb{N}} \subseteq \mathbb{R}_2.$
Therefore by Proposition \ref{IF} we get
\begin{eqnarray*}
f(x) &=&  \sum_{n=0}^{+\infty} x^n a_{n}=\sum_{n=0}^{+\infty} (\omega_{+}p+ \omega_{-}q)^n a_{n} \\
&=& \sum_{n=0}^{+\infty} (\omega_{+}p^n+ \omega_{-}q^n) (\omega_{+} b_{n}+ \omega_{-}c_{n}) \\
&=& \omega_{+}\sum_{n=0}^{+\infty} p^n b_{n}+ \omega_{-} \sum_{n=0}^{+\infty} q^n c_{n}.
\end{eqnarray*}
The last two series converge in $\{p \in \mathbb{R}_2: \, \, n(p)<R\}$ and $\{q \in \mathbb{R}_2: \, \, n(q)<R \}$, respectively.
\end{proof}

\begin{theorem}[Representation formula]
Let $D \subset \mathbb{C}$. Let us consider $f: \mathcal{Q}_{\mathbb{R}_3} \to \mathbb{R}_{3}$ a bi-slice regular function on $ \Omega_D$.  if we suppose that $f= \omega_{+}F+\omega_{-}G$. Then for all $I, W_1,K_1 \in \mathbb{S}_{\mathbb{R}_2}$ and $J,W_2, K_2 \in \mathbb{S}_{\mathbb{R}_2}$ we have
\begin{eqnarray*}
f(\omega_{+}p+ \omega_{-}q)\! \! \! \! \ &=&  \! \! \! \! \omega_{+}\bigl \{(W_1-K_1)^{-1}[W_1F(x+W_1 \cdot y)-K_1F(x+K_1\cdot y)]+I(W_1-K_1)^{-1}[F(x+W_1\cdot y)+\\
&&-F(x+K_1\cdot y)]\bigl \}+ \omega_{-} \bigl \{(W_2-K_2)^{-1}[W_2G(x+W_2 \cdot y)-K_2G(x+K_2\cdot y)]+\\
&& +J(W_2-K_2)^{-1}[G(x+W_2
\cdot y)-G(x+K_2 \cdot y)] \bigl \}\\
&=& \frac{\omega_{+}}{2}[F(x+W_1\cdot y)+F(x-W_1 \cdot y)]+ \frac{\omega_{-}}{2}[G(x+W_2 \cdot y)+G(x-W_2 \cdot y)]+\\
&&+  \frac{\omega_{+} IW_1}{2}[F(x+W_1\cdot y)-F(x-W_1 \cdot y))]+\frac{\omega_{-} JW_2}{2}[G(x+W_2 \cdot y)-G(x-W_2 \cdot y)] .
\end{eqnarray*}
\end{theorem}
\begin{proof}
By Theorem \ref{fund1} we know that a bi-slice regular function is a sum of two slice- regular functions $F$ and $G$. Therefore the statement follows by applying the quaternionic representation formula \cite[Thm. 1.15]{GSS} to each $F$ and $G$. The second equality follows by applying two times \cite[Cor.1.16]{GSS}.
\end{proof}
\begin{nb}
It is possible to write the representation formula in a matrix form
\begin{eqnarray*}
f(\omega_{+}p+ \omega_{-}q)&=& \frac{\omega_{+}}{2} (1 \quad I) \begin{pmatrix}
	1 && 1\\
	-W_1 && W_1
\end{pmatrix} \begin{pmatrix}
	F(x+W_1 \cdot y)\\
	F(x-W_1 \cdot y)
\end{pmatrix}+\\
&&+\frac{\omega_{-}}{2} (1 \quad J) \begin{pmatrix}
	1 && 1\\
	-W_2 && W_2
\end{pmatrix} \begin{pmatrix}
	G(x+W_2 \cdot y)\\
	G(x-W_2 \cdot y)
\end{pmatrix}
\end{eqnarray*}

\end{nb}
\begin{lemma}[Splitting lemma]
Let $U:=U' \times U'' \subseteq \mathcal{Q}_{\mathbb{R}_3}$ be a cartesian product of open sets. Let $f:U \to \mathbb{R}_{3}$ be a bi-slice regular function. Then for every $I,J \in \mathbb{S}_{\mathbb{R}_2}$, we can find $K, K' \in \mathbb{S}_{\mathbb{R}_2}$ such that $I \perp K$, $J \perp K'$. Then there exist four holomorphic functions
$$ A,B: U' \cap \mathbb{C}_I \to \mathbb{C}_I,$$
$$ C,D: U'' \cap \mathbb{C}_{J} \to \mathbb{C}_J,$$
such that for very $z=x+yI$ and $w=x+yJ$ we have
$$ f(z,w)= \omega_{+}  A(x+yI)+\omega_{-}  C(x+yJ)+ \omega_{+}B(x+yI) \cdot K +\omega_{-}D(x+yJ) \cdot K' .$$
\end{lemma}
\begin{proof}
Since $f$ is a bi-slice regular function, by Theorem \ref{fund1} we can write
\begin{equation}
\label{spl0}
f(x)= \omega_{+}F(p)+ \omega_{-}G(q), \qquad p=x+Iy \qquad q=x+Jy
\end{equation}
where $F:U' \to \mathbb{R}_2$ and $G:U'' \to \mathbb{R}_2$ are slice regular functions. The classic splitting Lemma \cite[Lemma 1.3]{GSS} holds for those functions, so for any $K\in \mathbb{S}_{\mathbb{R}_2}$ with $I \perp K$ there exists two holomorphic functions
$$ A,B: U' \cap \mathbb{C}_{I} \to \mathbb{C}_{I}$$
such that
\begin{equation}
\label{spl1}
F(x+Iy)= A(x+Iy)+B(x+Iy) \cdot K.
\end{equation}
The same reasoning for $G$ implies that for any $K' \in \mathbb{S}_{\mathbb{R}_2}$ with $ J \perp K'$ there exists two holomorphic functions
$$ C,D: U'' \cap \mathbb{C}_J \to \mathbb{C}_{J},$$
such that
\begin{equation}
\label{sp12}
G(x+Jy)=C(x+Jy)+D(x+Jy) \cdot K'.
\end{equation}
Putting \eqref{spl1} and \eqref{sp12} in \eqref{spl0} we get the statement.
\end{proof}

\section{Cauchy formula}
The bi-slice regular functions inherit a version of the Cauchy Theorem from the slice regular functions, see \cite[Chap. 6]{GSS}. Let us start with some notation.
\\From now on, $\gamma_I:[0,1] \to \mathbb{C}_{I}$ is a rectifiable curve whose support lies in a complex plane $ \mathbb{C}_I$ for some $I \in \mathbb{S}_{\mathbb{R}_{2}}$.
\begin{prop}[Cauchy Theorem]
Let $D \subset \mathbb{C}$. Let $f= \omega_{+}F+\omega_{-}G$ be a bi-slice regular function on a domain $ \Omega_D^4:= (\Omega_I)^2 \times (\Omega_J)^2$, where $(\Omega_I)^2:= \Omega_{D}^2 \cap \mathbb{C}_I$ and $(\Omega_J)^2:= \Omega_{D}^2 \cap \mathbb{C}_J$. Let $I,J \in \mathbb{S}_{\mathbb{R}_2}$ be such that $ \Omega_I$ and $ \Omega_J$ are simply connected and $ \gamma_I$ and $ \gamma_J$ are rectifiable closed curves in $ \Omega_I$, $ \Omega_J$, respectively. Then
$$ \int_{\gamma_I} ds F(s)=0,$$
$$ \int_{\gamma_J} ds' G(s')=0.$$
\end{prop}
\begin{proof}
By hypothesis $f$ is bi slice regular, this implies by Theorem \ref{fund1} that $F(q)$ and $G(p)$, with $q=x+Iy$ and $p=x+Jy$, are slice regular quaternionic functions. Therefore by applying \cite[Prop. 6.1]{GSS} to $F$ and $G$, respectively, we get the statement.
\end{proof}
It is also possible to prove a sort of Morera theorem
\begin{theorem}
Let $D \subset \mathbb{C}$. Let $ \Omega_D^4= (\Omega_D^2)' \times (\Omega_D^2)''$ be a domain in $ \mathcal{Q}_{\mathbb{R}_3}$. Let us consider a function $f= \omega_{+}F+ \omega_{-}G$ which goes from $ \Omega_D^4$ to $ \mathbb{R}_3$, If for every $I,J \in \mathbb{S}_{\mathbb{R}_2}$ the restrictions of $F$ and $G$ to $ (\Omega_D^2)_I:= (\Omega^2_D)'_{| \mathbb{C}_I}$ and $ (\Omega_D^2)_J=(\Omega^2_D)''_{| \mathbb{C}_I}$ , respectively, are continuous and  they satisfy
$$ \int_{\gamma_I} ds F(s)=0,$$
$$ \int_{\gamma_{J}} ds' G(s')=0.$$
Then for all rectifiable curves $ \gamma_I:[0,1] \to \mathbb{C}_I$ and $ \gamma_J: [0,1] \to \mathbb{C}_J$ the function $f$ is bi-slice regular in $ \Omega^4_D$.
\end{theorem}
\begin{proof}
Since $F$ and $G$ satisfy all the conditions of \cite[Prop. 6.2]{CSS} then $F$ and $G$ are slice regular functions in $ (\Omega_D^2)'$ and $(\Omega_D)''$, respectively. By the fact that $ \Omega_D^4= (\Omega_D^2)' \times (\Omega_D^2)''$ and by Theorem \ref{fund1} we get that $f$ is a bi-slice regular function in $ \Omega_D^4$.
\end{proof}
Also in this setting is possible to write a Cauchy formula.
\begin{theorem}[Cauchy formula]
Let $U$ be an axially s-domain. If $U=U'\times U''$. We set $U_{I}:= U'_{| \mathbb{C}_I}$ and $U_{J}:= U'_{| \mathbb{C}_J}$. We suppose that $ \partial(U_I \cap \mathbb{C}_I)$ and $ \partial (U_{I} \cap \mathbb{C}_J)$ are finite union of continuously differentiable Jordan curves for all $I, J \in \mathbb{S}_{\mathbb{R}_2}$. We consider $s=\omega_{+}s'+\omega_{-}s''$, where $s', s'' \in \mathbb{R}_2$.
We set $ds'_I=-ds'I$, $ds''_J=-ds''J$. If $f= \omega_{+}F+ \omega_{-}G$ is a bi-slice regular function on a set that contains $ \bar{U}$, then
$$ f(x)= \frac{\omega_{+}}{2 \pi} \int_{\partial(U' \cap \mathbb{C}_I)} S^{-1}(s',p) ds'_I f(s')+ \frac{\omega_{-}}{2 \pi} \int_{\partial(U'' \cap \mathbb{C}_J)} S^{-1}(s'',q) ds''_I f(s''),$$
where
$$ S^{-1}(s',q)=(q^2-2 \hbox{Re}(s') q+|s'|^2)^{-1}(\bar{s'}-q).$$
\end{theorem}
\begin{proof}
By Theorem \ref{fund1} we get that
\begin{equation}
\label{1s1}
f(x)= \omega_{+}F(p)+ \omega_{-}G(q), \qquad p=x+Iy, \qquad q=x+Jy,
\end{equation}
where $F$ and $G$ are slice regular functions. Since all the hypothesis of \cite[Thm. 4.5.3]{CSS} are satisfied we can write both $F$ and $G$, reactively, as
\begin{equation}
\label{s2}
F(p)=\frac{1}{2 \pi} \int_{\partial(U' \cap \mathbb{C}_I)} S^{-1}(s',p) ds'_I F(s'),
\end{equation}
\begin{equation}
\label{s3}
G(q)=\frac{1}{2 \pi} \int_{\partial(U'' \cap \mathbb{C}_J)} S^{-1}(s'',q) ds''_J F(s'').
\end{equation}
The thesis follows by putting \eqref{s2} and \eqref{s3} in \eqref{1s1}.
\end{proof}
Now, we focus on some properties of the bi-slice Cauchy kernel
$$ S^{-1}(s,x):= \omega_{+}S^{-1}(s',p)+ \omega_{-}S^{-1}(s'',q).$$
\begin{prop}
The bi-slice Cauchy kernel is left bi-slice regular in $x$ and right bi-slice regular in s, in its domain of definition.
\end{prop}
\begin{proof}
Let us consider $x= \omega_{+}p+ \omega_{-}q$ and $s= \omega_{+}s'+\omega_{-}s''=\omega_{+}(u+Iv)+ \omega_{-}(u+Jv)$, then
$$ \overline{\partial}_{IJ} S^{-1}(u+Iv, u+Jv,x)= \omega_{+}(\partial_u+\partial_v I)S^{-1}(s',p)+ \omega_{-}(\partial_u+ \partial_vJ)S^{-1}(s'',q)=0.$$
The last equality follows by \cite[Prop. 4.4.9]{CSS}. Similarly, if $s= \omega_{+}s'+ \omega_{-}s''$ and $x= \omega_{+}(x'+Iy')+ \omega_{-}(x'+Jy')$ we get
$$ \overline{\partial}_{IJ} S^{-1}(s,x'+Iy', x'+Jy')= \omega_{+}(\partial_x'+I\partial_y')S^{-1}(s',p)+ \omega_{-}(\partial_x'+ J \partial_y')S^{-1}(s'',q)=0.$$
Also in this case the last equality follows by \cite[Prop. 4.4.9]{CSS}.
\end{proof}
\begin{prop}
Let $p,q \in \mathbb{R}_2\setminus \mathbb{R}$. The singularities of the bi-slice Cauchy kernel lie on a 4-sphere.
\end{prop}
\begin{proof}
This follows by the fact the singularises of the  quaternionic Cauchy kernel lie on 2-sphere, see \cite[Prop. 4.4.9]{CSS}.
\end{proof}

\section{Zeros and multiplicity}
In this section we show how the splitting of $ \mathbb{R}_3$ has influence on the roots of a polynomial with values in $ \mathbb{R}_3$. Firstly, we prove how the $*$-product is splitted in the quaternionic components.
\begin{prop}
\label{inve0}
Let $x \in \mathcal{Q}_{\mathbb{R}_3}$. It is possible to write its inverse element as
\begin{equation}
\label{inve}
x^{-1}= \omega_{+} p^{-1}+ \omega_{-} q^{-1}
\end{equation}
where $p= x+Iy$ and $q=x+Jy$.
\end{prop}
\begin{proof}
In order to see that \eqref{inve} is the inverse element we have to show that $xx^{-1}=1$. Indeed, by the properties of $ \omega_{+}$ and $ \omega_{-}$ we get
$$ xx^{-1}=(\omega_{+}p+ \omega_{-}q) (\omega_{+} p^{-1}+ \omega_{-} q^{-1})= \omega_{+}+ \omega_{-}=1.$$
\end{proof}
It is possible to generalize the previous result.
\begin{prop}
\label{inve1}
Let $x \in \mathcal{Q}_{\mathbb{R}_3}$. Then we have
\begin{equation}
x^{-n}= \omega_{+} p^{-n}+ \omega_{-} q^{-n}, \qquad n \in \mathbb{N}
\end{equation}
where $p= x+Iy$ and $q=x+Jy$.
\end{prop}
\begin{proof}
We prove the result by induction. For $n=1$ we have proved the result in Proposition \ref{inve0}.
\\ Now, let us suppose that the statement holds for $n$, we will prove it for $n+1$. By the formula \eqref{inve} we get
\begin{eqnarray*}
x^{-n-1}&=&x^{-n}x^{-1}=(\omega_{+} p^{-n}+ \omega_{-} q^{-n})(\omega_{+} p^{-1}+ \omega_{-} q^{-1})\\
&=& \omega_{+} p^{-n-1}+ \omega_{-} q^{-n-1}.
\end{eqnarray*}
\end{proof}
\begin{theorem}
Let $x=\omega_{+}p+ \omega_{-} q$. Let us consider $f(x)= \omega_{+}A(p)+ \omega_{-}B(q)$, $f \neq 0$, $g(x)= \omega_{+}C(p)+ \omega_{-}D(q)$ bi-slice regular functions on the bi-axially symmetric domain $ \Omega \subseteq \mathcal{Q}_{\mathbb{R}_3}$, with $f$ tame, for all $x \in \Omega$ we have
$$ f*g(x)=\omega_{+} A(p)C\left(A^{-1}(p)p A(p)\right)+ \omega_{-}B(q)D\left(B^{-1}(q)q B(q)\right)$$
 \end{theorem}
\begin{proof}
We set $x= \omega_{+}p+ \omega_{-} q$. Since $f(x)= \omega_{+}A(p)+ \omega_{-}B(q)$ and $g(x)= \omega_{+}C(p)+ \omega_{-}D(q)$, by Proposition \ref{inve} and by the fact that $f$ is tame we get
\begin{eqnarray*}
f*g(x)&=&f(x) g\left(f(x)^{-1} x f(x) \right)\\
&=& \left(\omega_{+}A(p)+ \omega_{-}B(q)\right) g \left( \left(\omega_{+}A^{-1}(p)+ \omega_{-}B^{-1}(q)\right)(\omega_{+} p+ \omega_{-}q) \left(\omega_{+}A(p)+ \omega_{-}B(q)\right)\right)\\
&=&\left(\omega_{+}A(p)+ \omega_{-}B(q)\right) g \left(  \omega_{+} A^{-1}(p)p A(p)+ \omega_{-}B^{-1}(q)q B(q)\right)\\
&=& \left(\omega_{+}A(p)+ \omega_{-}B(q)\right)\left(\omega_{+}C\left(A^{-1}(p)p A(p)\right)+ \omega_{-} D\left(B^{-1}(q)q B(q)\right)\right)\\
&=& \omega_{+} A(p)C\left(A^{-1}(p)p A(p)\right)+ \omega_{-}B(q)D\left(B^{-1}(q)q B(q)\right).
\end{eqnarray*}
\end{proof}

\begin{lemma}
\label{propcs}
Let $f= \omega_{+}F+ \omega_{-}G$ be a function from $ \mathcal{Q}_{\mathbb{R}_3}$ in $ \mathbb{R}_3$ and $F$ ,$G$ be quaternionic-valued functions, then
\begin{itemize}
\item $f^c=\omega_{+}F^c+\omega_{-}G^c,$
\item $f^s= \omega_{+} F^s+\omega_{-}G^s.$
\end{itemize}
\end{lemma}
\begin{proof}
\begin{itemize}
\item
Let us write the function $f$ as
$$ f(x)= \sum_{n=0}^\infty x^n a_{n} \qquad \{a_n\}_{n \in \mathbb{N}} \subset \mathbb{R}_3.$$
Since we can write $f= \omega_{+}F+\omega_{-}G$ and by the splitting of the sequence as
$$ a_n=\omega_{+}b_n+\omega_{-}c_n, \qquad  \{b_n\}_{n \in \mathbb{N}}, \{c_n\}_{n \in \mathbb{N}} \subseteq \mathbb{R}_2$$
we get
$$ f^c(x)=\sum_{n=0}^\infty x^n a_{n}^c= \omega_{+} \sum_{n=0}^\infty p^n b_n^c+ \omega_{-}\sum_{n=0}^\infty q^n c_n^c= \omega_{+}F(p)+ \omega_{-}G(q).$$
\item From the previous point we get
\begin{eqnarray*}
\omega_{+} F^s+\omega_{-}G^s &=& \omega_{+}(F*F^c)+\omega_{-}(G*G^c)\\
&=& (\omega_{+}F* \omega_{+}F^c)+(\omega_{-}G* \omega_{-}G^c)\\
&=& (\omega_{+}F+ \omega_{-}G)*(\omega_{+}F^c+ \omega_{-}G^c)\\
&=& f* f^c=f^s.
\end{eqnarray*}
\end{itemize}
\end{proof}
\begin{nb}
If $A(p)=0$ and $B(q)=0$ (i.e. $f(x)=0$) we get that $ f*g(x)=0$.
\end{nb}
Now, we give a full study of the zeros of a polynomial with values in $ \mathbb{R}_3$.
\begin{theorem}
\label{fact}
Let us consider $x \in \mathcal{Q}_{\mathbb{R}_3}$, $x= \omega_{+}p+\omega_{-} q$, with $p=u+Iv$ and $q=u+Jv$. Moreover, let us assume that $ \alpha, \beta \in \mathbb{R}_3$, with $ \alpha= \omega_{+} a_1+ \omega_{-}a_2$ and $\beta= \omega_{+} b_1+ \omega_{-} b_2$, $ a_1,a _2, b_1,b_2 \in \mathbb{R}_2$. Let us consider the following function
$$ f(x)=(x- \alpha)*(x- \beta).$$
When $f$ is equal zero it is possible to split as
$$ \begin{cases}
(p-a_1)*(p-b_1)=0\\
(q-a_2)*(q-b_2)=0
\end{cases}
$$
\end{theorem}
\begin{proof}
First of all we write the function $f$ with the quaternionic components of $x$, $\alpha$ and $ \beta$, respectively.
\begin{eqnarray*}
f(x) &=& [(\omega_{+}p+\omega_{-} q)-(\omega_{+} a_1+ \omega_{-}a_2)]*[(\omega_{+}p+\omega_{-} q)-(\omega_{+} b_1+ \omega_{-} b_2)]\\
&=& [\omega_{+}(p-a_1)+ \omega_{-}(q-a_2)]*[\omega_{+}(p-b_1)+ \omega_{-}(q-b_2)].
\end{eqnarray*}
Now, we perform the $*$-product and by the properties of $ \omega_{+}$ and $ \omega_{-}$ we get
$$ f(x)= \omega_{+}(p-a_1)*(p-b_1)+\omega_{-}(q-a_2)*(q-b_2).$$
If the function $f(x)$ is equal to zero, by the orthogonality of $ \omega_{+}$ and $ \omega_{-}$ we get
$$ \begin{cases}
(p-a_1)*(p-b_1)=0\\
(q-a_2)*(q-b_2)=0
\end{cases}
$$
\end{proof}

\begin{nb}
Theorem \ref{fact} suggests that the factorization over  $\mathbb{R}_3$ is linked with that one in $ \mathbb{R}_2$.
\end{nb}
The previous result is fundamental to find the roots of polynomial in $ \mathbb{R}_3$. Let us start with some examples

\begin{example}
\label{ex1}
Let us consider $h(x)=(x-e_{12})*(x-e_{23}).$ In order to find the roots of this polynomial we split $e_{12}$ and $e_{23}$ in quaternionic components using the linear system \eqref{sys}. Thus
$$ e_{12}=\omega_{+} e_{12}+ \omega_{-} e_{12},$$
$$ e_{23}=\omega_{+} e_{23}+ \omega_{-} e_{23}.$$
We set $x= \omega_{+}p+ \omega_{-} q$ and by Theorem \ref{fact} we get
$$ \begin{cases}
(p-e_{12})*(p-e_{23})=0\\
(q-e_{12})*(q-e_{23})=0
\end{cases}
$$
This system implies that the roots are two double points in $ \mathbb{R}_2$, specifically $p=e_{21}$ and $q=e_{12}$.
\end{example}

\begin{example}
\label{ex2}
Let us consider $w(x)=(x-e_{1})*(x-e_{1}).$ In order to find the roots of this polynomial we split $e_{1}$ and $e_{2}$ in quaternionic components using the linear system \eqref{sys}. Thus
$$ e_{1}=-\omega_{+} e_{23}+ \omega_{-} e_{23},$$
We set $x= \omega_{+}p+ \omega_{-} q$ and by Theorem \ref{fact} we get
$$ \begin{cases}
(p+e_{23})*(p-e_{23})=0\\
(q+e_{23})*(q-e_{23})=0.
\end{cases}
$$
This means that
$$ \begin{cases}
p^2+1=0\\
q^2+1=0.
\end{cases}
$$
Therefore the roots of the polynomial are two 2-spheres.
\end{example}

\begin{example}
\label{ex3}
Let us consider $ g(x)=(x-e_{1})*(x-e_{23})$. Now, we split $e_1$ in quaternionic components using the linear system \eqref{sys}. Thus
$$ e_{1}=\omega_{+}(-e_{23})+ \omega_{-} e_{23}.$$
From Example \ref{ex1} we know that
$$ e_{23}=\omega_{+} e_{23}+ \omega_{-} e_{23}.$$
We set $x= \omega_{+}p+ \omega_{-} q$ and by Theorem \ref{fact} we get
$$ \begin{cases}
(p+e_{23})*(p-e_{23})=0\\
(q-e_{23})*(q-e_{23})=0
\end{cases}
$$
This means
$$ \begin{cases}
p^2+1=0\\
(q-e_{23})*(q-e_{23})=0
\end{cases}
$$
Therefore, the roots are a 2-sphere and a double point in $\mathbb{R}_2$, specifically $q= e_{23}$.
\end{example}

\begin{nb}
Using the splitting of an element of $ \mathcal{Q}_{\mathbb{R}_3}$ in a couple of quaternions it is possible to see further zeros with respect to \cite[Proposition 4.10]{GPS1}.
\end{nb}
Now, we give a full classification for the zeros of a polynomial of degree two. We call the set of the zeros of a function $f$ as $ \mathcal{Z}_f$.
\begin{theorem}
\label{l1}
Let us consider $x \in \mathcal{Q}_{\mathbb{R}_3}$, $x= \omega_{+}p+\omega_{-} q$, with $p=u+Iv$ and $q=u+Jv$. Furthermore, let us assume that $ \alpha, \beta \in \mathbb{R}_3$, with $ \alpha= \omega_{+} a_1+ \omega_{-}a_2$ and $\beta= \omega_{+} b_1+ \omega_{-} b_2$, $ a_1,a _2, b_1,b_2 \in \mathbb{R}_2$. we have all the following possibilities for the zeros of $f(x)=(x- \alpha)*(x-\beta)$
\begin{itemize}
\item[1)] If $a_{1}= b_1^c$, $a_2=b_2^c$
\begin{itemize}
\item[1.1)] and $a_1$, $a_2$ lie on the same sphere,
\item[1.2)] or $a_1$, $a_2$ lie on different spheres,
\end{itemize}
then $ \mathcal{Z}_f= \{(\mathbb{S}_{\mathbb{R}_2}^{a_1}, \mathbb{S}_{\mathbb{R}_2}^{a_2})\},$ which is  a 4-dimensional sphere.
\item[2)] If $a_{1}= b_1^c$, $a_2\neq b_2^c$ and $a_2$, $b_2$ lie on the same sphere then as points set we have $ \mathcal{Z}_f= \{(\mathbb{S}_{\mathbb{R}_2}^{a_1}, a_2) \}.$
\item[3)] If $a_{1} = b_1^c$ and $a_2$, $b_2$ do not lie on the same sphere then $$ \mathcal{Z}_f= \left\{(\mathbb{S}_{\mathbb{R}_2}^{a_1}, a_{2}),\left(\mathbb{S}_{\mathbb{R}_2}^{a_1},(b_2-a_2^c)^{-1}b_2(b_2-a_2^c)\right) \right\}.$$
\item[4)] $a_1$, $b_1$ do not lie on the same sphere as well as $a_2, b_2$ then
\begin{eqnarray*}
\mathcal{Z}_f&=& \biggl\{\left(a_{1}, a_2\right),\left(a_{1}, (b_2-a_2^c)^{-1}b_2(b_2-a_2^c) \right), \left((b_1-a_1^c)^{-1}b_1(b_1-a_1^c), a_{2}\right),\\
&& \left((b_1-a_1^c)^{-1}b_1(b_1-a_1^c), (b_2-a_2^c)^{-1}b_2(b_2-a_2^c) \right) \biggl\}.
\end{eqnarray*}

\item[5)]  If $a_{1} \neq b_1^c$ but the points lie on the same sphere and $a_2\neq b_2^c$ and the points lie on the same sphere as well, then as points set we have $ \mathcal{Z}_f= \{(a_1, a_2) \}.$
\item[6)] If $a_2 \neq b_2^c$ and lie on the same sphere, if $a_1$, $b_1$ do not lie on the same sphere then as points set we have $\mathcal{Z}_f= \{(a_1, a_2), \left((b_1-a_1^c)^{-1}b_1(b_1-a_1^c),a_2 \right) \}$
\end{itemize}
\end{theorem}
\begin{proof}
We show only the case $1.2$. By hypothesis $ \alpha= \beta^c$. Therefore by Theorem \ref{fact} we get
$$
\begin{cases}
(p-a_1)*(p-a_1^c)=0\\
(q-a_2)*(q-a_2^c)=0
\end{cases}
$$
If we develop the computations we get
$$
\begin{cases}
p^2-2 \hbox{Re}(a_1)p+|a_1|^2=0\\
q^2-2 \hbox{Re}(a_2)q+|a_2|^2=0.
\end{cases}
$$
These are the generic equations of two different quaternionic spheres.
\\It is possible to prove the other cases by combining Theorem \ref{fact} and \cite[Prop. 3.20]{GSS}.
\end{proof}
\begin{example}
Let us consider
$$ f(x)= (x-e_1)*(x-e_1). $$
As made in the Example \ref{ex2} we can split this polynomial as
$$ \begin{cases}
(p+e_{23})*(p-e_{23})=0\\
(q+e_{23})*(q-e_{23})=0.
\end{cases}
$$
We are in the case 1.1 of Theorem \ref{l1}. Therefore $ \mathcal{Z}_{f}= \{( \mathbb{S}_{\mathbb{R}_2}^{e_{23}}, \mathbb{S}_{\mathbb{R}_2}^{e_{23}}) \}.$
\end{example}
\begin{example}
Let us consider
$$ f(x)= (x-e_1)*(x-2e_1). $$
Similarly as Example \ref{ex2} we can write
$$ \begin{cases}
(p+e_{23})*(p-e_{23})=0\\
(q-2e_{23})*(q+2e_{23})=0.
\end{cases}
$$
We are in the case 1.2 of Theorem \ref{l1}. Therefore $ \mathcal{Z}_{f}= \{( \mathbb{S}_{\mathbb{R}_2}^{e_{23}}, \mathbb{S}_{\mathbb{R}_2}^{2e_{23}}) \}.$
\end{example}
\begin{example}
Let us consider
$$ g(x)=(x-e_1)*(x-e_{23}).$$
As made in the Example \ref{ex3} we can split the polynomial in the following way
$$ \begin{cases}
(p+e_{23})*(p-e_{23})=0\\
(q-e_{23})*(q-e_{23})=0.
\end{cases}
$$
We are in the case 2 of Theorem \ref{l1}. Therefore as points set we have $ \mathcal{Z}_{f}=  \{( \mathbb{S}_{\mathbb{R}_2}^{e_{23}}, e_{23} )\}.$
\end{example}

\begin{example}
Let us consider
$$ f(x)= (x-2e_{23})*(x+e_{23}-2e_{13}-e_{1}+2e_2).$$
Using the linear system \eqref{sys} and recalling that $x= \omega_{+}p+ \omega_{-}q$ we get
$$ \begin{cases}
(p-2e_{23})*(p+2e_{23})=0\\
(q-2e_{23})*(q-4e_{13})=0.
\end{cases}$$
We are in the case 3 of Theorem \ref{l1}. Then $ \mathcal{Z}_f= \left\{(\mathbb{S}_{\mathbb{R}_2}^{2 e_{23}},2 e_{23}), \left(\mathbb{S}_{\mathbb{R}_2}^{2 e_{23}},\frac{4}{5}(4e_{23}+3e_{13})\right)\right\}.$
\end{example}

\begin{example}
Let us consider the polynomial
$$ f(x)=(x-2e_{23}+e_{1})*(x-4e_{13}-2e_2).$$
We can split the polynomial in the following quaternionic equations
$$ \begin{cases}
(p-3e_{23})*(p-6e_{13})=0\\
(q-e_{23})*(q-e_{13})=0.
\end{cases}
$$
We are in the case 4 of Theorem \ref{l1}. Then after some computations we have
\begin{eqnarray*}
\mathcal{Z}_f &=& \biggl\{(3e_{23},e_{23}), \left(3 e_{23},\frac{3}{5}(7 e_{13}-6 e_{23})\right), \left(\frac{4}{3}(3e_{13}-e_{23}), e_{23}\right), \left(\frac{4}{3}(3e_{13}-e_{23}), \frac{3}{5}(7 e_{13}-6 e_{23})\right) \biggl \}.
\end{eqnarray*}
\end{example}

\begin{example}
Let us consider the polynomial
$$f(x)=(x-e_{12})*(x-e_{23}).$$
As made in Example \ref{ex1} we have
$$ \begin{cases}
(p-e_{12})*(p-e_{23})=0\\
(q-e_{12})*(q-e_{23})=0.
\end{cases}
$$
We are in the case 5 of Theorem \ref{l1}. Therefore as points set we have $ \mathcal{Z}_f= \{(e_{12}, e_{12})\}.$
\end{example}
\begin{example}
Let us consider the polynomial
$$ f(x)=(x-e_{12}-e_{13}+e_3+e_2)*(x-e_{12}-2e_{13}+e_3+2e_2).$$
We split it as
$$ \begin{cases}
(p-2e_{12})*(p-2e_{12})=0\\
(q-2e_{13})*(q-4e_{13})=0.
\end{cases}
$$
We are in the case 6 of Theorem \ref{l1}. Therefore as points set we have $ \mathcal{Z}_{f}= \{(2e_{12},2 e_{13}),(2e_{12},-4e_{13})\}.$
\end{example}
It is possible to define a concept of multiplicity inherited by the classical one \cite[Sec. 3.6]{GSS} and  also other new types.

\begin{Def}
Let $f=\omega_{+}F+ \omega_{-}G$ be a polynomial with values in $ \mathbb{R}_3$. Let $x,y \in \mathbb{R}$, with $y \neq 0$. We say that $f$ has four-dimensional spherical multiplicity $2n+2m$ at $\left(x+y \mathbb{S}_{\mathbb{R}_2}, x+y \mathbb{S}_{\mathbb{R}_2} \right)$ if $m$ and $n$ are the largest number such that
$$ \begin{cases}
F(p)= [(p-x)^2+y^2]^m \widetilde{F}(p)\\
G(q)= [(q-x)^2+y^2]^n\widetilde{G}(q),
\end{cases}
$$
where $ \widetilde{F}$ and $\widetilde{G}$ are other polynomials such that $ \widetilde{F}(x+Iy) \neq 0$ and $ \widetilde{G}(x+Jy) \neq 0.$
\end{Def}

\begin{Def}
If the polynomials $ \widetilde{F}$ and $\widetilde{G}$, previously defined, have one root $p_1 \in x+y \mathbb{S}_{\mathbb{R}_2}$ and $q_{1} \in x+y \mathbb{S}_{\mathbb{R}_{2}},$ respectively, then the polynomial $f$ has isolated
multiplicity $n+m$ at $(p_1,q_1)$ if and only if there exist $p_2,..., p_n \in x+y \mathbb{S}_{\mathbb{R}_2}$, with $p_i \neq p_{i+1}^c$ for all $ i \in \{1,...,n-1\}$, there exist $q_2,..., q_m \in x+y \mathbb{S}_{\mathbb{R}_2}$, with $q_j \neq q_{j+1}^c$ for all $ j \in \{1,...,m-1\}$ and
$$ \begin{cases}
\widetilde{F}(p)=(p-p_1)*...*(p-p_n)* R(p)\\
\widetilde{G}(q)=(q-q_1)*...*(q-q_m)* R(q),
\end{cases}
$$
where $R(q)$ and $R(p)$ are two regular polynomials.
\end{Def}

Now, we have two new concepts of multiplicity.

\begin{Def}
Let $f=\omega_{+}F+ \omega_{-}G$ be a polynomial with values in $ \mathbb{R}_3$. Let $x,y \in \mathbb{R}$, with $y \neq 0$. If the polynomial $\widetilde{G}$, previously defined, has one root $q_1 \in x+y \mathbb{S}_{\mathbb{R}_2}$, then $f$ has two-dimensional spherical multiplicity of the first kind $2n+m$ at( $ \mathbb{S}_{\mathbb{R}_2}$, $q_1$) if and only if there exist $q_2,..., q_m \in x+y \mathbb{S}_{\mathbb{R}_2}$, with $q_i \neq q_{i+1}^c$ for all $ i \in \{1,...,m-1\}$, and $n$ is the largest natural number such that
$$ \begin{cases}
F(p)= [(p-x)^2+y^2]^n \widetilde{F}(p)\\
\widetilde{G}(q)=(q-q_1)*...*(q-q_m)* R(q),
\end{cases}
$$
where $ \widetilde{F}$ is other polynomials such that $ \widetilde{F}(x+Iy) \neq 0$ and $R(q)$ is a regular polynomial.
\end{Def}
\begin{Def}
Let $f=\omega_{+}F+ \omega_{-}G$ be a polynomial with values in $ \mathbb{R}_3$. Let $x,y \in \mathbb{R}$, with $y \neq 0$. If the polynomial $\widetilde{F}$, previously defined, has one root $p_1 \in x+y \mathbb{S}_{\mathbb{R}_2}$, then $f$ has two-dimensional spherical multiplicity of the second kind $n+2m$ at ($p_1$, $ \mathbb{S}_{\mathbb{R}_2}$) if and only if there exist $q_2,..., q_m \in x+y \mathbb{S}_{\mathbb{R}_2}$, with $p_i \neq p_{i+1}^c$ for all $ i \in \{1,...,n-1\}$, and $m$ is the largest natural number such that
$$ \begin{cases}
\widetilde{F}(p)=(p-p_1)*...*(p-p_m)* R(p)\\
G(q)= [(q-x)^2+y^2]^n \widetilde{G}(q),
\end{cases}
$$
where $ \widetilde{G}$ is a polynomial such that $ \widetilde{G}(x+Jy) \neq 0$ and $R(q)$ is a regular polynomial.
\end{Def}

\begin{prop}
\label{a1}
If $p(x)$ is a bi-regular polynomial of degree $d$, then the sum of the spherical multiplicities and isolated multiplicities of the zeros of $p(x)$ is $2d$.
\end{prop}
\begin{proof}
Since we can split every bi-regular polynomial as a sum of two regular polynomial of the same degree of the polynomials $p(x)$, and from the fact that each regular polynomial has the property that the sum of the spherical multiplicities and isolated multiplicities is equal to the degree, \cite[Prop. 3.20]{GSS}, we get the statement.
\end{proof}

\begin{theorem}{[Fundamental Theorem of Algebra]}
\label{a2}
Any polynomial $P(x): \mathcal{Q}_{\mathbb{R}_3} \to \mathbb{R}_3$  with degree $d$ has at least a roots, which is a couple of quaternions. This means that it has at least a root in $ \mathbb{R}_3$.
\end{theorem}
\begin{proof}
We can write the polynomial $P(x)$ as sum of two quaternionic polynomials of the same degree, let su denote them with $F$ and $G$. For these polynomials hold the fundamental Theorem of Algebra, see \cite[Thm. 3.18]{GSS}, therefore both $F$ and $G$ have at least a root in $ \mathbb{R}_2$.
\end{proof}
\begin{nb}
Proposition \ref{a1} and Theorem \ref{a2} improve the result obtained in \cite[Theorem 26]{GP}.
\end{nb}
Now, we show some examples.
\begin{example}
Let us consider the following polynomial $h(x)= \frac{(x-e_{12}-e_{23}+e_3+e_1)*(x-e_{12}-e_{23}+e_3+e_1)}{4}$. By Example \ref{ex1} we can split this polynomial as
$$ \begin{cases}
(p-e_{12})^{*2}=0\\
(q-e_{23})^{*2}=0.
\end{cases}
$$
The four-dimensional spherical multiplicity is 0 at $(\mathbb{S}_{\mathbb{R}_2}, \mathbb{S}_{\mathbb{R}_2})$, the isolated multiplicity is 4 at $(e_{12}, e_{23})$ and the two kinds of two-dimensional spherical multiplicities are 2 at ($ \mathbb{S}_{\mathbb{R}_2}, e_{23})$ and $( e_{12}, \mathbb{S}_{\mathbb{R}_2})$.
\end{example}

\begin{example}
Let us consider the following polynomial $g(x)= (x-e_{1})*(x-e_{1})$. By Example \ref{ex2} we can split this polynomial as
$$ \begin{cases}
p^2+1=0\\
q^2+1=0.
\end{cases}
$$
The roots are two 2-spheres, so the four-dimensional spherical multiplicity is 4 at $(\mathbb{S}_{\mathbb{R}_2}, \mathbb{S}_{\mathbb{R}_2})$, the isolated one is 0 at $(0,0)$
and the two kinds of two-dimensional spherical multiplicities are 2 in $(\mathbb{S}_{\mathbb{R}_2},0)$ and $(0,\mathbb{S}_{\mathbb{R}_2})$, respectively.
\end{example}

\begin{example}
Let us consider the following polynomial $h(x)= (x-e_{1})*(x-e_{23})$. By Example \ref{ex3} we can split this polynomial as
$$ \begin{cases}
p^2+1=0\\
(q-e_{23})^{*2}=0.
\end{cases}
$$
The roots are double points and a  2-sphere, the four-dimensional spherical multiplicity is 2 at $(\mathbb{S}_{\mathbb{R}_2}, \mathbb{S}_{\mathbb{R}_2})$, the isolated is 2 at $(0,e_{23})$, the two-dimensional spherical multiplicity of the first kind is 4 at $(\mathbb{S}_{\mathbb{R}_2}, e_{23})$ and  the two-dimensional spherical multiplicity of the second kind is 0 at $(0, \mathbb{S}_{\mathbb{R}_2})$.
\end{example}

\section{Further application: Determinant of a matrix with entries in the quadratic cone $ \mathcal{Q}_{\mathbb{R}_3}$}
In this section we show how the peculiar structure of the quadratic cone $ \mathcal{Q}_{\mathbb{R}_3}$ and the splitting of $ \mathbb{R}_{3}$ can be used for computing a formula for the determinant of a matrix with entries in the quadratic cone

As in the quaternionic case, see \cite{BG1}, 
the determinant of matrix with entries in $\mathcal{Q}_{\mathbb{R}_3}$ cannot be defined as in the case of matrices with real or complex entries. Therefore there is a need to find a new formula.
\\We denote by $ \mathcal{M}(2, \mathcal{Q}_{\mathbb{R}_3})$ the $\mathcal{Q}_{\mathbb{R}_3}$-vector space right or left of $2 \times 2$ matrices with entries in $ \mathcal{Q}_{\mathbb{R}_3}$.
\begin{Def}
A matrix $ \begin{pmatrix}
             a & b \\
             c & d
\end{pmatrix} \in \mathcal{M}(2, \mathcal{Q}_{\mathbb{R}_3})$ is right invertible if and only if there exists
$ \begin{pmatrix}
x & y \\
t & z
\end{pmatrix}
 \in \mathcal{M}(2, \mathcal{Q}_{\mathbb{R}_3})$
 such that
 $$ \begin{pmatrix}
             a & b \\
             c & d
\end{pmatrix}
\begin{pmatrix}
x & y \\
t & z
\end{pmatrix}
= \begin{pmatrix}
1 & 0 \\
0 & 1
\end{pmatrix}.$$
\end{Def}
It is possible to prove the following result by using the same tecnicque of \cite[Prop. 2.1]{BG1}.
\begin{prop}
\label{df}
The following statements are equivalent
\begin{itemize}
  \item[1)]  The matrix $ \begin{pmatrix}
             a & b \\
             c & d
\end{pmatrix} \in \mathcal{M}(2, \mathcal{Q}_{\mathbb{R}_3})$ is right invertible;
  \item[2)] $b(c-db^{-1}a) \neq 0$ or $a(d-ca^{-1}b) \neq 0$;
  \item[3)] $c(b-ac^{-1}d) \neq 0$ or $d(a-bd^{-1}c) \neq 0$.
\end{itemize}
\end{prop}
Now, we want to compute
$$ n \left( a(d-ca^{-1}b) \right), \qquad a,b,c,d \in \mathcal{Q}_{\mathbb{R}_3}.$$
In order to compute this norm we want to split the following elements as sum of two quaternions. Let us consider
\begin{equation}
\label{d1}
a= \omega_{+} a'+ \omega_{-}a'',
\end{equation}

\begin{equation}
\label{d2}
b= \omega_{+} b'+ \omega_{-}b'',
\end{equation}

\begin{equation}
\label{d3}
 c= \omega_{+} c'+ \omega_{-}c'',
\end{equation}

\begin{equation}
\label{d4}
d= \omega_{+} d'+ \omega_{-}d'',
\end{equation}
with $a'a'',b',b'',c',c'',d',d'' \in \mathbb{R}_2$, with $a', a''$ having the same real parts and same modulus of the imaginary parts, as well as the other quaternions. By Proposition \ref{inve} we get
\begin{eqnarray*}
ca^{-1}b &=& (\omega_{+} c'+ \omega_{-}c'') (\omega_{+} (a')^{-1}+ \omega_{-}(a'')^{-1})(\omega_{+} b'+ \omega_{-}b'') \\
&=& \omega_{+}c'(a')^{-1}b'+ \omega_{-} c''(a'')^{-1}b''.
\end{eqnarray*}
Therefore
$$ a(d-ca^{-1}b)= \omega_{+} \left(a'(d'-c'(a')^{-1}b') \right)+\omega_{-} \left(a''(d''-c''(a'')^{-1}b'') \right).$$
Therefore by Remark \ref{norm} and \cite[Lemma 2.4]{BG1} we get
\begin{eqnarray*}
n \left( a(d-ca^{-1}b) \right) &=& n \left(a''(d''-c''(a'')^{-1}b'') \right)=n \left(a'(d'-c'(a')^{-1}b') \right) \\
&=& |a'||d'|+|c'||b'|-2 \hbox{Re}(d'\bar{b}' a' \bar{c}') \\
&=& |a''||d''|+|c''||b''|-2 \hbox{Re}(d''\bar{b}'' a'' \bar{c}'').
\end{eqnarray*}
By \cite[Lemma 2.4]{BG1} we get the following
\begin{lemma}
\label{df1}
Let us consider $a,b,c, d \in \mathcal{Q}_{\mathbb{R}_3}$, decomposed as in \eqref{d1},\eqref{d2}, \eqref{d3}, \eqref{d4}. Then we have
\begin{eqnarray*}
n\left(a(d-ca^{-1}b)\right) &=& n\left(b(c-db^{-1}a)\right)=n\left(c(b-ac^{-1}d)\right)=n\left(d(a-bd^{-1}c)\right) \\
&=& |a'||d'|+|c'||b'|-2 \hbox{Re}(d'\bar{b}' a' \bar{c}') \\
&=& |a''||d''|+|c''||b''|-2 \hbox{Re}(d''\bar{b}'' a'' \bar{c}'').
\end{eqnarray*}
\end{lemma}
\begin{nb}
\label{df2}
Let $ q \in \mathbb{R}_2$ be a generic quaternion. Due to the fact that $ \hbox{Re} (q) \leq |q|$ we have
$$ |a'||d'|+|c'||b'|-2 \hbox{Re}(d'\bar{b}' a' \bar{c}') \geq (|a'||d'|-|b'||c'|)^2 \geq 0,$$
and analogously
$$ |a''||d''|+|c''||b''|-2 \hbox{Re}(d''\bar{b}'' a'' \bar{c}'') \geq (|a''||d''|-|b''||c''|)^2 \geq 0,$$
\end{nb}
Proposition \ref{df}, Lemma  \ref{df1} and Remark \ref{df2} lead to the following
\begin{Def}
\label{det}
If $A= \begin{pmatrix}
a & b \\
c & d
\end{pmatrix} \in \mathcal{M}(2, \mathcal{Q}_{\mathbb{R}_3})$ and each entries decomposed as in \eqref{d1},\eqref{d2}, \eqref{d3}, \eqref{d4}, then the determinant is defined as
\begin{eqnarray}
 \hbox{det}_{\mathcal{Q}_{\mathbb{R}_3}}(A)&=& \sqrt{|a'||d'|+|c'||b'|-2 \hbox{Re}(d'\bar{b}' a' \bar{c}')}\\
 &=& \sqrt{|a''||d''|+|c''||b''|-2 \hbox{Re}(d''\bar{b}'' a'' \bar{c}'')}.
\end{eqnarray}
\end{Def}
\begin{nb}
If the entries of the matrix $A$ are real or complex, then the definition of the determinant coincides with the classical notion of determinant.
\end{nb}
\begin{nb}
If $a,b,c,d \in \mathcal{Q}_{\mathbb{R}_3}$ and they are splitted as \eqref{d1}, \eqref{d2}, \eqref{d3}, \eqref{d4} then
\begin{eqnarray*}
A= \begin{pmatrix}
a & b \\
c & d
\end{pmatrix} &=& \begin{pmatrix}
\omega_{+} a'+ \omega_{-}a'' &  \omega_{+} b'+ \omega_{-}b'' \\
 \omega_{+} c'+ \omega_{-}c'' &  \omega_{+} d'+ \omega_{-}d''
\end{pmatrix}  \\
&=& \omega_{+}\begin{pmatrix}
a' & b' \\
c' & d'
\end{pmatrix}
+ \omega_{-} \begin{pmatrix}
a' & b' \\
c' & d'
\end{pmatrix}\\
&:=& \omega_{+} \tilde{A}+ \omega_{-}\tilde{\tilde{A}}
\end{eqnarray*}
Therefore
$$ A= \omega_{+}\tilde{A}+ \omega_{-}\tilde{\tilde{A}}.$$
Hence by Definition \ref{det} and \cite[Def. 2.6]{BG1} we get
\begin{equation}
\label{fdet}
\hbox{det}_{\mathcal{Q}_{\mathbb{R}_3}}A= \hbox{det}_{\mathbb{R}_2}(\tilde{A})=\hbox{det}_{\mathbb{R}_2}(\tilde{\tilde{A}}).
\end{equation}
\end{nb}
\begin{example}
We want to compute the determinant of the following matrix
$$ A= \begin{pmatrix}
e_1 & e_{2}+e_{23} \\
-1 & e_{2}
\end{pmatrix}.
$$
First of all split the entries in quaternionic components. By the linear system \eqref{sys} we get
$$ e_{1}= - \omega_{+}e_{23}+ \omega_{-}e_{23},$$
$$ e_{2}+e_{23}= \omega_{+}(e_{13}+e_{23})+ \omega_{-}(e_{13}-e_{23}),$$
$$ -1= -\omega_{+}- \omega_{-},$$
$$ e_{13}= \omega_{+} e_{13}- \omega_{-}e_{13}.$$
Thus
$$ \begin{pmatrix}
e_1 & e_{2}+e_{23} \\
-1 & e_{2}
\end{pmatrix}= \omega_{+} \begin{pmatrix}
-e_{23} & e_{13}+e_{23} \\
-1 & e_{13}
\end{pmatrix}+\omega_{-}\begin{pmatrix}
e_{23} &  e_{2}-e_{23} \\
-1 & -e_{13}
\end{pmatrix}.$$
By Definition \ref{det} we get
$$ \hbox{det}_{\mathcal{Q}_{\mathbb{R}_3}}A= \sqrt{3}.$$
We observe that the value of the determinant is equal to that one obtained by computing the determinant of the two quaternionic matrices.
\end{example}
Now, we show some properties of the determinant.
\begin{lemma}
Let us consider $ \lambda, \mu \in \mathcal{Q}_{\mathbb{R}_3}$. We write their decomposition as $ \lambda= \omega_{+}\lambda'+ \omega_{-} \lambda''$, $ \mu= \omega_{+} \mu'+ \omega_{-}\mu''$, where $\lambda', \lambda''$  are quaternions with the same real parts and modulus of imaginary parts, as well as $\mu', \mu''$. Let $A =\begin{pmatrix}
a & b \\
c & d
\end{pmatrix}
$
be a matrix in $ \mathcal{M}(2, \mathcal{Q}_{\mathbb{R}_3})$ we have
\begin{itemize}
\item[1)]  $\hbox{det}_{\mathcal{Q}_{\mathbb{R}_3}} \begin{pmatrix}
a & b \lambda \\
c & d \lambda
\end{pmatrix}=\hbox{det}_{\mathcal{Q}_{\mathbb{R}_3}} \begin{pmatrix}
a\lambda & b \\
c\lambda & d
\end{pmatrix}= | \lambda'| \hbox{det}_{\mathcal{Q}_{\mathbb{R}_3}}\begin{pmatrix}
a & b \\
c & d
\end{pmatrix}=| \lambda''| \hbox{det}_{\mathcal{Q}_{\mathbb{R}_3}}\begin{pmatrix}
a & b \\
c & d
\end{pmatrix} $
\item[2)] $\hbox{det}_{\mathcal{Q}_{\mathbb{R}_3}} \begin{pmatrix}
a \mu & b \mu \\
c & d
\end{pmatrix}=\hbox{det}_{\mathcal{Q}_{\mathbb{R}_3}} \begin{pmatrix}
a  & b  \\
c \mu & d \mu
\end{pmatrix}= | \mu'| \hbox{det}_{\mathcal{Q}_{\mathbb{R}_3}}\begin{pmatrix}
a & b \\
c & d
\end{pmatrix} =| \mu''| \hbox{det}_{\mathcal{Q}_{\mathbb{R}_3}}\begin{pmatrix}
a & b \\
c & d
\end{pmatrix}$
\item[3)] If the matrix $B$ is obtained from the matrix $A$ by:
\begin{itemize}
\item[3.1)] substituting to a row the sum of the two rows, or
\item[3.2)] substituting to a column the sum of the two columns
\end{itemize}
then $\hbox{det}_{\mathcal{Q}_{\mathbb{R}_3}}(A)=\hbox{det}_{\mathcal{Q}_{\mathbb{R}_3}}(B).$
\end{itemize}
\end{lemma}

\begin{prop}[Binet property]
For all $A,B \in \mathcal{M}(2, \mathcal{Q}_{\mathbb{R}_3})$ we have
$$ \hbox{det}_{\mathcal{Q}_{\mathbb{R}_3}}(AB)=\hbox{det}_{\mathcal{Q}_{\mathbb{R}_3}}(A)\hbox{det}_{\mathcal{Q}_{\mathbb{R}_3}}(B).$$
\end{prop}
\begin{proof}
Firstly, we decompose each entries of the matrices $A,B$ in quaternionic components to get
$$ A= \omega_{+} \tilde{A}+ \omega_{-} \tilde{\tilde{A}},$$
$$ B= \omega_{+} \tilde{B}+ \omega_{-} \tilde{\tilde{B}},$$
where $\tilde{A}$, $\tilde{\tilde{A}}$, $\tilde{B}$, $\tilde{\tilde{B}}$ are matrices with quaternionic entries. Making the multiplication between the two matrices we get
$$AB= \omega_{+} \tilde{A}\tilde{B}+ \omega_{-}\tilde{\tilde{A}}\tilde{\tilde{B}}.$$
By formula \eqref{fdet} and \cite[Prop. 2.10]{BG1} we get
\begin{eqnarray*}
 \hbox{det}_{\mathcal{Q}_{\mathbb{R}_3}}(AB)&=& \hbox{det}_{\mathbb{R}_2}(\tilde{A}\tilde{B}) \\
&=& \hbox{det}_{\mathbb{R}_2}(\tilde{A}) \hbox{det}_{\mathbb{R}_2}(\tilde{B}) \\
&=& \hbox{det}_{\mathcal{Q}_{\mathbb{R}_3}}(\tilde{A}) \hbox{det}_{\mathcal{Q}_{\mathbb{R}_3}}(\tilde{B}).
\end{eqnarray*}
\end{proof}

\hspace{2mm}

\noindent
Cinzia Bisi,
Dipartimento di Matematica e Informatica \\Universit\`a di Ferrara\\
Via Ma\-chia\-vel\-li n.~30\\
I-44121 Ferrara\\
Italy

\noindent
\emph{email address}: bsicnz@unife.it\\
\emph{ORCID iD}: 0000-0002-4973-1053

\vspace*{5mm}
\noindent
Antonino De Martino,
Dipartimento di Matematica \\ Politecnico di Milano\\
Via Bonardi n.~9\\
20133 Milan\\
Italy

\noindent
\emph{email address}: antonino.demartino@polimi.it\\
\emph{ORCID iD}: 0000-0002-8939-4389

\vspace*{5mm}

\end{document}